\documentclass{book}
\usepackage[contrib,lang=american]{ems-book} 

\usepackage{amscd,dsfont} 
\usepackage{mathrsfs}

\newcommand{\R}{{\mathord{\mathbb R}}}
\newcommand{\one}{{\mathds1}}
\def\C{\mathbb{C}} 
\newcommand{\Z}{{\mathord{\mathbb Z}}}
\newcommand{\N}{{\mathord{\mathbb N}}}

\newcommand{\cA}{\mathord{\mathcal{A}}} 

\newcommand{\HH}{{\mathcal{H}}}
\newcommand{\tr}{\operatorname{Tr}} 
\numberwithin{equation}{section}
\newcommand{\cH}{\mathcal{H}}
\newcommand{\cK}{\mathcal{K}}

\newcommand{\cB}{\mathcal{B}}

\theoremstyle{plain}
    \newtheorem{thm}{Theorem}[chapter]
    \newtheorem{lm}[thm]{Lemma}
    \newtheorem{cl}[thm]{Corollary}

\theoremstyle{plain}
    \newtheorem{defi}[thm]{Definition}
    \newtheorem{exam}[thm]{Example}
    \newtheorem{remark}[thm]{Remark}

\begin{document}

\mainmatter

\title{On some convexity and monotonicity inequalities of Elliott Lieb}
\titlemark{Convexity and monotonicity inequalities} 

\emsauthor{1}{Eric A. Carlen}{E.~A.~Carlen}


\emsaffil{1}{Department of Mathematics, Hill Center, Rutgers University,
110 Frelinghuysen Road
Piscataway NJ 08854-8019 USA\email{carlen@math.rutgers.edu}}

\classification[65F60]{39B62}

\keywords{Convexity, Monotonicity, Trace Inequalities}

\begin{abstract}
A fundamental paper of Elliott Lieb from 1973 has been the basis for much beautiful work on matrix inequalities by many people over the following years. We review a well-connected set of these developments. Some new proofs are provided.

\medskip
\centerline{\emph{Dedicated to Elliott Lieb on the occasion of his 90th birthday.}}

\end{abstract}

\makecontribtitle

\section{Introduction}   Elliott Lieb's work on matrix inequalities has inspired a great many mathematicians and physicists, and the starting point of this paper 
is an account of some of the developments relating to the first three theorems in his fundamental 1973 paper \textit{Convex trace functions and the 
Wigner-Yanase-Dyson conjecture}.  Lieb's paper contains many results, and not only did it solve what has become known as the Wigner-Yanase-Dyson conjecture,  it provided the mathematical tools that were used to resolve another important conjecture  made by Lanford and Robinson that was  explicitly written out in \cite{LR68}, the {\em strong subadditivity of the quantum mechanical entropy} (SSA). 
This conjecture was proved by Lieb and Ruskai \cite{LR}, also in 1973. 

However, \cite{L73} also  contained theorems that would answer questions that nobody had yet asked, and in some cases, 
when {\em years later} the questions were asked, it was not recognized that the answers were to be found in Lieb's paper \cite{L73}.  
The present paper begins by discussing  Theorems 2 and 3 of \cite{L73}, along with Theorem 1, the 
Lieb Concavity Theorem, which proved  the Wigner-Yanase-Dyson conjecture.  These three theorems were all stated  
as convexity or concavity theorems, but they have equivalent formulations in terms of monotonicity of certain matrix functionals under 
certain classes of completely positive maps. These formulations have turned out to be important to many more recent applications of these results. 

These three theorems, as well as two more in the same paper, have been the starting point for many other developments since 1973.  We cannot go into all of these developments here -- that would require a much longer paper. Instead, we will focus on a selection of developments with a fairly close relation to the first three theorems in \cite{L73}, eventually arriving at the latter two.   This paper is essentially expository although there are more than a few new and  simpler proofs.

The following notation will be used throughout this paper. $M_n(\C)$ denotes the set of $n\times n$ complex matrices, $M_n^+(\C)$ the set of positive semidefinite  $n\times n$ complex matrices, and 
$M_n^{++}(\C)$ the set of positive definite  $n\times n$ complex matrices.   In some cases it is convenient to refer instead to bounded operators on a Hilbert space $\cH$. To keep the proofs simple, we shall always assume that $\cH$ is finite dimensional. 
This is the heart of the matter.   The reader familiar with the necessary technical tools will readily see how to make extensions. For the generalization to a general von Neumann algebra setting that also makes a connection with the matrix algebra case treated here, see Araki's paper \cite{A75}, and the references it contains. 

The paper consists of four parts. Part One consists of Sections 2 and 3,  in which  the theorems that are our focus are introduced in both their convexity/concavity forms and their monotonicity forms. Some simple but important consequences are noted, and the equivalence of their  convexity/concavity and monotonicity formulations is proved. 

Part Two consists of Sections 4 through 8, in which a number of applications of the theorems from Part One are made. All of the applications discussed here have some sort of connection with SSA, for which several different proofs are given. 

Part Three consists of sections 9 through 11, and these recall some well known material on the Lieb-Ruskai Monotonicity Theorem,  the theory of operator monotone and convex functions, the application of the GNS representation to proving matrix inequalities.    

Part Four consists of sections 12 and 13, and these are devoted  -- finally -- to the proof of monotonicity  theorems discussed in Section 2. Section 12 in particular presents an important theorem of Hiai and Petz \cite{HP12} whose proof is amazingly simple, and which yields all of the first three theorems in \cite{L73} as special cases. 

\section{The first three theorems in Lieb's 1973 matrix inequality paper and their monotonicity formulations}

The first three theorems in \cite{L73} are:

\begin{thm}[Lieb, Theorem 1]\label{L1}
For $0 \leq s,t$ and $s+t \leq 1$, and any fixed $K\in M_n(\C)$, the function
\begin{equation}\label{lieb1}
(X,Y) \mapsto \tr[K^*Y^s K X^t]
\end{equation}
is jointly concave on $M_n^+(\C)\times M_n^+(\C)$. 
\end{thm}

The special case $s=t =1/2$ had been proved earlier by Wigner and Yanase \cite{WY1,WY2}; see the remarks at the end of the section for more information. 

\begin{thm}[Lieb, Theorem 2] \label{L2} For all $s,t\geq 0$, $s+t \leq 1$, 
\begin{equation}\label{lieb2} 
(X,Y,K) \mapsto    \tr[ K^*Y^{-s}  KX^{-t}]
\end{equation} 
is jointly convex on $M_n^{++}(\C)\times M_n^{++}(\C)\times M_n(\C)$.  
\end{thm}

\begin{thm}[Lieb, Theorem 3]\label{L3}
\begin{equation}\label{lieb3} 
(X,Y,K) \mapsto    \tr \left[   \int_0^\infty K^*\frac{1}{sI + Y} K \frac{1}{sI + X}{\rm d}s \right]\ 
\end{equation}
is jointly convex on $M_n^{++}(\C)\times M_n^{++}(\C)\times M_n(\C)$.
\end{thm}

Lieb proved Theorem~\ref{L1} using an interesting interpolation argument 
and he  then deduced Theorem~\ref{L2} and Theorem~\ref{L3} from it by variational arguments. 
However, in the final section of \cite{L73} he proved that these three theorems -- along with 
two others  -- are ``equivalent'' to one another, meaning that once any one of them has been proved, short, elementary  arguments yield each of the others. 
Lieb also proved certain corollaries of these theorems; e.g.,  in Theorem~\ref{L1}, one may replace the right side of \eqref{lieb1} by  $(\tr[K^*Y^s K X^t])^{1/q}$ 
for all $q$ such that $0 <q(s+t) \leq 1$, and the concavity still holds. However, this generalization comes rather easily once one has Theorem~\ref{L1}, and we shall not use it in any of the applications discussed here.

In Theorem~\ref{L1}, the Lieb Concavity Theorem, the case in which $s = 1-t$  is particularly important. In fact, one can easily deduce the general case from this special case, as was pointed out by Araki \cite{A75}: Suppose $s < 1-t$ and let $r$ be such that $r(1-t) = s$. Then for positive invertible $X_1,X_2,Y_1,Y_2$,
$$K^*(Y_1+Y_2)/2)^sK =  K^*(((Y_1+Y_2)/2)^r)^{1-t} K \geq K^* ((Y_1^r+Y^r_2)/2)^{1-t} K $$
where we have used the well known fact that  of $Y \mapsto Y^r$ is monotone increasing  and concave in $Y$ for $0 < r < 1$. (See \cite{C10} for the standard integral representation that provides a simple proof.)
Now take the trace against $(X_1+X_2)/2)^t$ to conclude from the $s=1-t$ case of \eqref{lieb1} that the general case is valid. 
In the same way, one readily shows that \eqref{lieb2} follows in general from the special case in which $s= 1-t$. 

An even more special case of \eqref{lieb1} arises when $K$ is the identity $I$, and $X$ and $Y$ are density matrices; i.e., $\tr[X] = \tr[Y] =1$. The 
{\em Umegaki relative entropy of $X$ with respect to $Y$} \cite{U62}, 
$D(X||Y)  := \tr[X(\log X - \log Y)]$,  
satisfies
\begin{equation}\label{lieb4}
 D(X||Y)  =  \frac{{\rm d}}{{\rm d}t} \tr[Y^{1-t}X^t]\bigg|_{t=1} = \lim_{t\uparrow 1}\frac{1}{1-t}\left(1 -  \tr[Y^{1-t}X^t]\right)\ .
\end{equation}
By the Lieb Concavity Theorem, for each $0 < t<1$, the right of \eqref{lieb4}  side is jointly convex in $X$ and $Y$, and this yields a
 result  of Lindblad \cite[Lemma 2]{Lind74}, who used exactly this differentiation argument:

\begin{thm}[Lindblad]\label{Lind}
The map
\begin{equation}\label{lieb5}
(X,Y)\mapsto D(X||Y)
\end{equation}
is jointly convex on $M_n^+(\C)\times M_n^+(\C)$
\end{thm}

There are several useful notions of  relative entropy for two density matrices, but the one that comes up most frequently is the one introduced by Umegaki, and in the following, when we refer to relative entropy using this short phrase, we mean the Umegaki relative entropy. 

Lindblad  had been working on proving Thereom~\ref{Lind},  and recognized that \cite{L73} provided the solution.    A much simpler but still useful fact  about the relative entropy is the non-negativity of $D(X||Y)$ when $\tr[X] = \tr[Y]$
For all $X,Y\in M_n^+(\C)$ with  $\tr[X] = \tr[Y]$
\begin{equation}\label{klein}
D(X||Y) \geq 0 \quad{\rm and}\quad D(X||Y) = 0 \ \iff\ X =Y\ .
\end{equation}
This is an easy consequence of Klein's inequality \cite{K31}, an elementary proof  of which may be found in Appendix A of \cite{CL19}, though we shall give another very simple proof of \eqref{klein} in Section 10.

Theorem~\ref{L2} has an important consequence:

\begin{thm}[Ando's Convexity Theorem]\label{ACT}
For all $K\in M_n(\C)$, and all  $1 \le q\le 2$ and $0\le r \le 1$ with $q-r\ge  1$,
\begin{equation}\label{act1}
(X,Y) \mapsto \tr(K^*X^{q}KY^{-r})
\end{equation}
is convex on $M_n^{+}(\C)\times M_n^{++}(\C)$.
\end{thm}

\begin{proof}  Note that
$$
\tr(K^*X^{q}KY^{-r}) = \tr((XK)^*X^{q-2}(XK)Y^{-r})
$$
and that $0 \leq (2-q),r  \leq 1$ and $(2-q) + r = 2-(q-r) \leq 1$ . Now apply Theorem~\ref{L2}.
\end{proof} 

Theorem~\ref{ACT} was proved in \cite{A79} by a different argument. The connection with Theorem~\ref{L2} was not noted. 

There are by now many simple proofs of the Lieb Concavity Theorem, and hence of Theorems~\ref{L1}, \ref{L2} and \ref{L3}; see for instance \cite{A11,E09,NEE}. However, a 2012 theorem of Hiai and Petz \cite{HP12} that is stated and proved in Section 12 is definitive in this direction:  It has a very simple proof that directly yields the monotonicity versions of these theorems, as we discuss next, and much more, all at once.

The convexity/concavity theorems, Theorems~\ref{L1}, \ref{L2}, \ref{L3} and  \ref{Lind},  all have {\em equivalent} formulations as {\em monotonicity theorems}, as we now explain, and not only have these been important to a number of their more recent applications, it turns out to be easiest to directly prove the monotonicity variants, and then to deduce the original convexity/concavity results from them.

 Let $\Phi:M_m(\C) \to M_n(\C)$ be a linear transformation. $\Phi$ is {\em positive} when $\Phi(A) \geq 0$ for all $A \geq 0$. $\Phi$ is $2$-positive if the block matrix
 $\left[\begin{array}{cc} \Phi(A) & \Phi(B) \\\Phi(C )& \Phi(D)\end{array}\right] \geq 0$ whenever  $\left[\begin{array}{cc} A & B \\C & D\end{array}\right] \geq 0$.  For each integer $k > 2$, the condition of $k$-positivity is defined in the analogous manner, and $\Phi$ is {\em completely} positive if it is $k$ positive for all $k$.   For example, it is easy to see that for any $m\times n$ matrix $V$,
 the map $\Phi: X \mapsto V^* XV$ is completely positive, and it follows that for any set $\{V_1,\dots,V_\ell\}$ of $m\times n$ matrices,
\begin{equation}\label{lieb6}
 \Phi: X \mapsto \sum_{j=1}^\ell V_j^* X V_j\ .
\end{equation}
 is completely positive. By a theorem of Kraus \cite{K71} and Choi \cite{Choi72,Choi75}, every completely positive map $\Phi$ from $M_m(\C)$ to $M_n(\C)$ has this form, and in fact, one can always find such a representation with $\ell \leq mn$.  The map $\Phi$ is {\em unital} if it takes the identity to the identity; i.e., $\Phi(I) = I$.  The map $\Phi$ is {\em trace preserving} if $\tr[\Phi(X)] = \tr[X]$ for all $X$.  For each $n$, we equip $M_n(\C)$ with the Hilbert-Schmidt inner product, making it a Hilbert space. For any linear transformation $\Phi:M_m(\C) \to M_n(\C)$, we write $\Phi^\dagger$ to denote the adjoint with respect to the Hilbert-Schmidt inner product. It is easy to see that if $\Phi$ is given by \eqref{lieb6}, then 
 \begin{equation}\label{lieb7}
 \Phi^\dagger : X \mapsto \sum_{j=1}^\ell V_j X V_j^*\ .
\end{equation}
In any case, whether $\Phi$ is completely positive or not, $\Phi$ is unital if and only if $\Phi^\dagger$ is trace preserving. 

Unital completely positive maps $\Phi$ play a fundamental role in quantum information theory and the quantum theory of open systems \cite{D76}; these maps  update the observables in the time evolution and partial measurement of such a system \cite{K71}.  The state of a finite dimensional quantum system may be identified with a {\em density matrix}, i.e., an $X\in M_n^+(\C)$ with $\tr[X] = 1$. If $\Phi$ is completely positive and unital, its adjoint 
 $\Phi^\dagger$ is  a completely positive trace preserving map, and so evidently it takes density matrices to density matrices.  Such maps $\Phi^\dagger$  are known  as {\em quantum operations}.   
 
 
 The following example is  fundamental:  Let $m,n\in \N$ and  identify $\C^{mn}$ with the $m$-fold direct sum of $m$ copies of $\C^n$. We may then think of matrices in $M_{mn}(\C)$  as $m\times m$ block matrices each of whose entries is in $M_n(\C)$.  Then, suppressing the $n$ dependence, for each $m$ we define
\begin{equation}\label{contomon1I}
\Psi_m(X) = \left[\begin{array}{ccc} X & \phantom{X} &  \phantom{X}\\   \phantom{X} & \ddots &  \phantom{X} \\  \phantom{X} &  \phantom{X} & X\end{array}\right]\ ,
\end{equation}
where the matrix on the right is the $m\times m$ block diagonal  matrix each of whose diagonal entries is $X$.  This is evidently completely positive and  unital.   Its adjoint, $\Psi_m^\dagger$, is therefore trace preserving and completely positive. It is easy to see that 
\begin{equation}\label{contomon1BI}
\Psi^\dagger_m\left(\left[\begin{array}{ccc} 
X_{1,1} & \cdots & X_{1,m}\\ \vdots & \ddots &\vdots\\
X_{m,1} & \cdots & X_{m,m}\end{array}\right]\right) = \sum_{j=1}^m X_{j,j}\ ,
\end{equation} is the {\em partial trace}.   
Some of the theorems that follow involve negative powers. When $\Phi:M_n(\C) \to M_m(\C)$ is unital and completely positive, $\Phi(A)\in M_m^{++}(\C)$ whenever $A\in M_n^{++}(\C)$,  However, the image of $M_m^{++}(\C)$ under the completely positive trace preserving map $\Phi^\dagger$  may lie entirely in $M_n^{+}(\C) \backslash M_n^{++}(\C)$.  For example, take $n=2m$ and think of elements of $M_n(\C)$ as $2\times 2$ block matrices with entries in $M_m(\C)$. Then  define
$$
\Phi\left(\left[\begin{array}{cc} A & B\\ C & D\end{array}\right]\right) = A  \qquad{\rm so \ that}\qquad \Phi^\dagger(A) =  \left[\begin{array}{cc} A & 0\\ 0 & 0\end{array}\right]\ .
$$
We may now state the monotonicity versions of  Theorems~\ref{L1}, \ref{L2}, \ref{L3} and \ref{Lind}.

\begin{thm}[Monotonicity version of Theorem~\ref{L1}]\label{L1M}   For all $0 \leq t \leq 1$, all $m,n\in \N$, all $X,Y \in M_m^+(\C)$, all $K\in M_n(\C)$,  and all completely positive unital maps $\Phi: M_n(\C)\to M_m(\C)$,  
\begin{equation}\label{lieb21}
 \tr[\Phi(K^*)Y^{1-t} \Phi(K) X^t]  \leq  \tr[K^*\Phi^\dagger(Y)^{1-t} K \Phi^\dagger(X)^t] \ .
\end{equation}
\end{thm}

\begin{thm}[Monotonicity version of Theorem~\ref{L2}]\label{L2M}   For all $0 \leq t \leq 1$, all $m,n\in \N$, all 
$X,Y \in M_m^{++}(\C)$ and all $K\in M_m(\C)$ ,  and all completely positive unital maps $\Phi: M_n(\C)\to M_m(\C)$ such that $\Phi^\dagger(X),\Phi^\dagger(Y) \in M_n^{++}(\C)$
\begin{equation}\label{lieb22} 
 \tr[ \Phi^\dagger(K^*)\Phi^\dagger(Y)^{t-1}  \Phi^\dagger(K)\Phi^\dagger(X)^{-t}]  \leq   \tr[ K^*Y^{t-1}  KX^{-t}]\ .
\end{equation} 
\end{thm}

\begin{thm}[Monotonicity version of Theorem~\ref{L3}]\label{L3M}   For all $0 \leq t \leq 1$, all $m,n\in \N$, all $X,Y \in M_m^{++}(\C)$, all $K\in M_m(\C)$ and all completely positive unital maps $\Phi: M_m(\C)\to M_m(\C)$, 
\begin{equation}\label{lieb23} 
 \tr \left[   \int_0^\infty  \Phi^\dagger(K^*)\frac{1}{sI + \Phi^\dagger(Y)} \Phi^\dagger(K) \frac{1}{sI + \Phi^\dagger(X)}{\rm d}s \right]   
  \leq   \tr \left[   \int_0^\infty K^*\frac{1}{sI + Y} K \frac{1}{sI + X}{\rm d}s \right]\ .
\end{equation}
\end{thm}

\begin{thm}[Monotonicity version of Theorem~\ref{Lind}]\label{LindM}   For all $m,n\in \N$, all $X,Y \in M_m^{+}(\C)$, and all completely positive unital maps $\Phi: M_n(\C)\to M_m(\C)$,  
\begin{equation}\label{lieb8}
 D(\Phi^\dagger(X)||\Phi^\dagger(Y))  \leq D(X||Y)\ .
\end{equation}
\end{thm}

The inequality \eqref{lieb8} is known as  the {\em Data Processing Inequality}, and is due to Lindblad \cite{Lind74}.  It implies that all quantum operations 
performed on density matrices $X$ and $Y$ reduce their relative entropy, and this has the effect of making the states harder to 
experimentally distinguish.  It is one of the cornerstones of quantum information theory; see the introduction of \cite{CFL18} for more information.
The first proof of Theorem~\ref{L1M} in this form is due to Uhlmann \cite[Proposition 17]{Uh77}.

The passage from Theorems~\ref{L1M}, \ref{L2M}, \ref{L3M}, \ref{LindM} and \ref{ACTM} back to Theorems~\ref{L1}, \ref{L2}, \ref{L3}, \ref{Lind} and \ref{ACT}
is easy to explain: In each case, one takes $\Phi^\dagger$ to be the partial trace map $\Psi_2^\dagger$ where $\Psi_m$ is defined in \eqref{contomon1I}.  For instance, suppose we have proved Theorem~\ref{L1M}. 
Let $X_1$, $X_2$, $Y_1$ and $Y_2$ be positive definite in $M_m(\C)$, and define
$$
X = \left[\begin{array}{cc} X_1 & 0 \\ 0 & X_2\end{array}\right]\qquad{\rm and}\qquad Y = \left[\begin{array}{cc} Y_1 & 0 \\ 0 & Y_2\end{array}\right]\ .
$$
Then for any $K\in M_n(\C)$, 
$$
 Y^{1-t}\Psi_2(K)X^t  =   \left[\begin{array}{cc} Y_1^{1-t}K X_1^t & 0 \\ 0 & Y_2^{1-t}KX_2^t\end{array}\right]
$$
so that $\tr[\Psi_2(K^*) Y^{1-t}\Psi_2(K)X^t ] = \tr[K^*Y_1^{1-t}K X_1^t] + \tr[K^*Y_2^{1-t}KX_2^t]$.  On the other hand,
$$\tr[ K^*\Psi_2^\dagger(Y)^{1-t}K \Psi_2^\dagger(X)^t ] = \tr[K^*(Y_1+Y_2)^{1-t} K (X_1+X_2)^t]\ .$$
Then by \eqref{lieb21}
\begin{equation*}
 \tr[K^*Y_1^{1-t}K X_1^t] + \tr[K^*Y_2^{1-t}KX_2^t]  \leq   \tr[K^*(Y_1+Y_2)^{1-t} K (X_1+X_2)^t]\ .
\end{equation*} 
Since $(X,Y)\mapsto \tr[K^*Y^{1-t}K X^t]$ is homogeneous of degree one, this is equivalent to \eqref{lieb1} for $s= 1-t$, and we have already explained how this implies the general case.  This proves that Theorem~\ref{L1} is a simple consequence of Theorem~\ref{L1M}, and  this argument is easily adapted to  the other  pairs of theorems. 

The passage from Theorems~\ref{L1}, \ref{L2}, \ref{L3} and  \ref{Lind}  to  Theorems~\ref{L1M}, \ref{L2M}, \ref{L3M} and  \ref{LindM}  is 
almost as simple. It relies on the Stinespring factorization \cite{S55} of any completely positive trace preserving map $\Phi^\dagger$ into the composition of three 
especially simple completely positive trace preserving maps, the third of which is the partial trace.
Then, a very useful lemma of  Uhlmann \cite{U71} that expresses partial traces as averages over unitary 
conjugations allows the convexity or concavity to be applied, and yields the monotonicity theorems. This is explained in the next section in complete detail. All of the arguments are elementary. 

One might expect, in analogy with Theorem~\ref{L2M}, that there would be a monotonicity variant of the Ando Concavity Theorem \ref{ACT} stating that 
for all $0 \leq t \leq 1$, all $m,n\in \N$,  all completely positive unital maps $\Phi: M_n(\C)\to M_m(\C)$,  
all $Y \in M_m^+(\C)$ and $X\in M_m^{++}(\C)$ with $\Phi^\dagger(X) \in M_n^{++}(\C)$, and all $K\in M_n(\C)$, 
\begin{equation}\label{andwro}
 \tr[\Phi(K^*)Y^{1+t} \Phi(K) X^{-t}]  \geq  \tr[K^* \Phi^\dagger(Y)^{1+t} K \Phi^\dagger(X)^{-t}] \ .
\end{equation}
{\em However, \eqref{andwro} is false}.  Consider the completely positive map $\Phi: M_n(\C) \to M_n(\C)$ given by $\Phi(A) = \frac1n \tr[A]I$ . Note that this is both unital and trace preserving, 
and $\Phi^\dagger = \Phi$.   Choosing $K\neq 0$ with $\tr[K] =0$, the left side of \eqref{andwro} is zero, while the right side is $\frac1n \tr[K^*K]( \tr[Y])^{1-t} (\tr[X])^{-t} > 0$.
However, we do have monotonicity under partial traces, which is shown to follow from the convexity in the next section:

\begin{thm}[Restricted monotonicity version of Theorem~\ref{ACT}]\label{ACTM}   For all $0 \leq t \leq 1$, all $m,n\in \N$,  
all $Y \in M_{mn}^+(\C)$ and $X\in M_{mn}^{++}(\C)$,  with $\Psi_m$ given by \eqref{contomon1I},
\begin{equation}\label{lieb21b}
 \tr[\Psi_m(K^*)Y^{1+t} \Psi_m(K) X^{-t}]  \geq  \tr[K^* \Psi_m^\dagger(Y)^{1+t} K \Psi_m^\dagger(X)^{-t}] \ .
\end{equation}
\end{thm}

We close this section with some remarks on the work of Wigner and Yanase \cite{WY1,WY2}. A 1960 paper of Araki and Yanase \cite{AY60} showed that the extent to which a quantum  observable fails to commute with a conserved quantity limits the precision with which it can be measured.  A few years later, when Yanase was visiting the I.A.S. in Princeton, he and Wigner investigated the implications for measuring the amount of information in a quantum state $\rho$; i.e., a density matrix. They presented a list of axioms that a good ``measure of information'' should satisfy, and in the case of a quantum mechanical density matrix $\rho$ and a self-adjoint  operator $H$, representing some conserved quantity such as the energy, they proposed \cite[Equation 2]{WY1} what they called the {\em skew information}:
$$
I(\rho,H) :=  -\frac12 \tr[[\rho^{1/2},H]^2]  = \tr[H^2\rho] - \tr[H\rho^{1/2} H \rho^{1/2}]
$$
Wigner and Yanase proved that for fixed $H$, $\rho \mapsto I(\rho,H)$ is convex, and this is evidently equivalent to the concavity of $\rho \mapsto 
\tr[H\rho^{1/2} H \rho^{1/2}]$. 
At first glance, this may look like a special case of \eqref{lieb1} for $s=t= \tfrac12$, but notice that  since the convexity proved by Wigner and Yanase does not depend on the normalization of $\rho$, we may consider arbitrary $X,Y\in M_n^{++}$ and $K \in M_n(\C)$ and form the block matrices
$$
\rho := \left[ \begin{array} {cc} Y & 0\\ 0 & X\end{array}\right] \quad{\rm and}\quad H :=  \left[ \begin{array} {cc} 0 & K\\ K^* & 0\end{array}\right]
$$
and then 
\begin{equation}\label{passageback}
\tr[H\rho^{1/2} H \rho^{1/2}]  =2 \tr [ K^* Y^{1/2}  KX ^{1/2}]\ ,
\end{equation}
and hence the convexity result of Wigner and Yanase is equivalent to the special case $s=t=\frac12$ of the Lieb Concavity Theorem.  

In the final paragraph of \cite{WY1}, Wigner and Yanase wrote that they were not sure that their proposed measure of information is the only one satisfying the axioms they specified, and they reported that they had also considered  $ -\frac12 \tr[[\log \rho,H][\rho,H]]$.  They also remark that Freeman Dyson pointed out  that these cases fell into a one parameter family of candidates
\begin{equation} -\frac12 \tr[[\rho^t,H][\rho^{1-t},H]]\ ,
\end{equation}
$0 < t \leq 1/2$ since
$$
\tr[[\log \rho,H][\rho,H]]  = \lim_{t\downarrow 0} \frac1t \tr[[\rho^t,H][\rho^{1-t},H]]\ .
$$
However, it is not clear that Wigner, Yanase or even Dyson conjectured anything about concavity or convexity in these cases, 
and Wigner and Yanase wrote that these other candidates have ``undesirable'' features, and appear to have been  dismissive of their further study. 

Nonetheless, Res Jost gave a partially alternate proof \cite{J70} of the theorem of Wigner and Yanase that was written during a visit he made to the I.A.S. in Princeton during Fall 1968. He thanks his friend Freeman Dyson for many discussions on the subject, and mentions the generalized cases proposed by Dyson. Jost does not ascribe any explicit conjecture to Dyson, but it certainly would appear that five years after the work of Wigner and Yanase, Dyson felt that his proposed functionals were worthy of investigation.

\section{From Convexity or Concavity  to Monotonicity}

A theorem of Stinespring \cite{S55} provides a factorization of any unital completely positive map $\Phi$ as a composition of three particularly simple completely positive maps.  
One of the factors has already been introduced: For $m\in \N$, let $\Psi_m$ be defined by \eqref{contomon1I} so that $\Psi_m^\dagger$ is defined by 
\eqref{contomon1BI}; i.e., so that $\Psi_m^\dagger$ is the partial trace.

Next,  for $p \leq q$, let $\Xi_{q,p}$ be the map from $M_q(\C)\to M_p(\C)$ that sends $Y\in M_q(\C)$ to upper left $p\times p$ block of $Y$. 
Then evidently, $\Xi^\dagger_{q,p}$ is the map sending $X\in M_p(\C)$ to the matrix in $M_q(\C)$ whose upper left $p\times p$ block is $X$, with all other entries being zero. 
 It is easy to write down the  Kraus representation of $\Xi_{q,p}$, and hence  $\Xi_{q,p}^\dagger$. 
 Let $W:\C^p \to \C^q$ be linear transformation by which the first $p$ entries of $W(v)$ are the same as those of $v$, in the same order, 
 and the rest, if any, are zero. Then for all $X\in M_p(\C)$,   $\Xi_{q,p}(X) = W^*XW$.  Evidently,  
 $\Xi_{q,p}$ is completely positive and unital, and $\Xi^\dagger_{q,p}$ is completely positive and trace preserving.

The following is a version of the Stinespring Representation Theorem for  unital and completely positive maps in the finite dimensional case. It is well-known to people who work on quantum information theory, but I do not know of an accessible written reference in this convenient form. For a thorough discussion of the standard operator-algebraic formulation, see \cite{Pa03}. 
 
\begin{thm}[Stinespring Representation]\label{strep}  Let $\Phi:M_p(\C)  \to M_q(\C)$ be completely positive and unital. Then there is a natural number $m$ such that $mp \geq q$, and a unitary $U\in M_{mp}(\C)$ such that for all $X\in M_p(\C)$
\begin{equation}\label{stine1}
\Phi(X) =    \Xi_{mq,q}  ( U^*\Psi_m(X)U )\ ,
\end{equation}
and consequently for all $Y\in M_q(\C)$, 
\begin{equation}\label{stine2}
\Phi^\dagger(Y) =   \Psi_m^\dagger( U\Xi_{mp,q}^\dagger(Y) U^* )\ .
\end{equation}
\end{thm}

In other words, by \eqref{stine1} every unital completely positive map from $M_p(\C)$  to $M_q(\C)$ acting on $X\in M_p(\C)$ consists of  {\em first} sending 
$X$ to the block diagonal $m\times m$ matrix  each of whose diagonal entries is $X$,   {\em second} applying a unitary conjugation to the 
$mp \times mp$ matrix,  {\em and third} picking off the upper left $q\times q$ corner of the result.  

Likewise, by \eqref{stine2} every trace preserving completely positive map from $M_q(\C)$  to $M_p(\C)$ acting on $Y\in M_q(\C)$ consists of {\em first} embedding $Y$ as the upper left $q\times q$ corner of matrix in $M_{mp}(\C)$  whose other entries are all zero, {\em second}, applying a unitary conjugation, and the regarding the result as an $m\times m$ block matrix with entries in $M_p(\C)$, {\em and third}, taking the partial trace.

The following lemma due to Uhlmann \cite{Uh73} connects partial traces with convexity. 

\begin{lm}\label{ullm} There is an explicit set $\{U_1,\dots,U_{m^2}\}$ of unitaries in $M_{mn}(\C)$ such that for all $Y \in  M_{mn}(\C)$, 
\begin{equation}\label{contomon1C}
\frac{1}{m^2}\sum_{k=1}^{m^2} U_k^*YU_k = \frac{1}{m} \left[\begin{array}{ccc} \Psi_m^\dagger(Y)  & \phantom{X} &  \phantom{X}\\   \phantom{X} & \ddots &  \phantom{X} \\  \phantom{X} &  \phantom{X} & \Psi_m^\dagger(Y)\end{array}\right] =  \frac1m \Psi_m(\Psi_m^\dagger(Y)) = \frac1m I_{\C^m}\otimes \Psi_m^\dagger(Y)\ ,
\end{equation}
where we have used the obvious identification of $\C^{mn}$ with $\C^m \otimes \C^n$. 
\end{lm}

We provide a simple proof for completeness.

\begin{proof}  The second and third equalities in \eqref{contomon1C}  are evident, and we need only prove the first. Write vectors  $v \in \C^{mn}$ in the form 
$v = (v_1,\dots,v_m)$ where each entry is in $\C^n$. 
For each $j=1,\dots,m$, let  $P_j$ be the orthogonal projection on $\C^{mn}$  leaves the $\ell$th entry of $v = (v_1,\dots,v_m)$ as it  is, and sends the others to $0$. That is, 
$P_1(v) = (v_1,0, \dots 0)$, $P_2(v) = (0,v_2,0,\dots,)$, etc.  
For $k=1,\dots,m$, define ${\displaystyle U_k =  \sum_{j=1}^m e^{i2\pi jk/m}J_j}$.
Since $J_k J_\ell = \delta_{k,\ell} P_\ell$, each $U_k$ is unitary. Moreover,
$$
\frac{1}{m}\sum_{k=1}^m U_k^*Y U_k =    \sum_{j,\ell=1}^m \left(\frac1m \sum_{k=1}^m (e^{i2\pi (j-\ell)/m})^k\right) P_\ell Y P_j =  \sum_{j,\ell=1}^m \delta_{j,\ell}P_\ell Y P_j  = \sum_{j=1}^m P_j YP_j \ .
$$
That is,
$$
\frac{1}{m}\sum_{k=1}^m U_k^*\left(  \left[\begin{array}{ccc} 
Y_{1,1} & \cdots & Y_{1,m}\\ \vdots & \ddots &\vdots\\
Y_{m,1} & \cdots & Y_{m.m}\end{array}\right]   \right) U_k  = 
\left[\begin{array}{ccc} Y_{1,1} & \phantom{X} &  \phantom{X}\\   \phantom{X} & \ddots &  \phantom{X} \\  \phantom{X} &  \phantom{X} & Y_{m,m}\end{array}\right]
$$
where the right hand side denotes the block diagonal matrix whose $j$th diagonal block is $Y_{j,j}$.  Now for $1 \leq \ell \leq m$ define $C_\ell$ to be the $\ell$th cyclic permutation matrix acting on the $m$-fold direct sum of $\C^m$ by
$$C_j(v_1,\dots, v_m) = (v_j,\dots,v_m,v_1,\dots)\ .$$
Then evidently the $m^2$ unitary matrices of the form $C_\ell U_k$, $1 \leq k,\ell \leq m$ have the desired property.
\end{proof}

\begin{remark}\label{inv}  Suppose $Y\in M_{mn}(\C)$ is of the from $Y = \Psi_m(A) = 
\left[\begin{array}{ccc} A & & \\ & \ddots & \\ & & A\end{array}\right]$. Then each of the 
$m^2$ unitaries
$U_k$ in Lemma~\ref{ullm} commutes with $Y$.  One way to see this is directly from the definitions -- it is easy to check that for each 
$1 \leq j \leq m$, the matrices $J_j$ and $C_j$ used to constructed the unitaries commute with such $Y$.  
Another is to note that for such $Y$,  we have $Y = m^{-2} \sum_{k=1}^{m^2}U_k^*YU_k$, so that
$$\tr\left[\left(m^{-2} \sum_{k=1}^{m^2}U_k^*YU_k\right)^*\left(m^{-2} \sum_{k=1}^{m^2}U_k^*YU_k\right)\right]  = \tr[Y^*Y] =  
m^{-2} \sum_{k=1}^{m^2}\tr[ (U_k^*YU_k)^*(U_k^*YU_k)]\ .$$
By the strict convexity of $Y \mapsto \tr[W^*W]$, $U_k^*YU_k$ is independent of $k$, and in our collection of unitaries is $I = C_mU_m$. 
\end{remark}

With these tools in hand, we proceed to deduce Theorems \ref{L1M}, \ref{L2M}, \ref{L3M}, \ref{LindM} and \ref{ACTM} from their convexity/concavity counterparts. It is simplest to begin with the relative entropy. 

\begin{proof}[Proof of Theorem~\ref{LindM}] We use the evident {\em additivity} of the relative entropy, namely that for diagonal block matrices
$A := \left[\begin{array}{cc} A_1 & \\ & A_2\end{array}\right]$ and  $B := \left[\begin{array}{cc} B_1 & \\ & B_2\end{array}\right]$ with $A_1,B_1 \in M_m^{+}(\C)$ and $A_2,B_2 \in M_n^{+}(\C)$
$$
D(A||B) = D(A_1||B_1) + D(A_2||B_2)\ ,
$$
and the fact that the relative entropy is homogeneous of degree one. 

Therefore, for all $X,Y \in  M_q^+(\C)$, 
\begin{eqnarray*}
 D(\Psi_m^\dagger(X) \ ||\ \Psi_m^\dagger(Y)) &=& D\left(\frac{1}{m} \Psi_m \Psi_m^\dagger(X)  \ \bigg|\bigg| \   \frac{1}{m} \Psi_m \Psi_m^\dagger(Y) \right) \\
 &=& D\left( \frac{1}{m^2}\sum_{k=1}^{m^2} U_k^*XU_k \ \bigg|\bigg| \   \frac{1}{m^2}\sum_{k=1}^{m^2} U_k^*YU_k\right) \\
 &\leq& \frac{1}{m^2}\sum_{k=1}^{m^2}D(U_k^*XU_k || U_k^*YU_k) = D(X||Y)\ .
\end{eqnarray*}
The first equality is the additivity and homogeneity discussed above, the second it Uhlmann's Lemma. The inequality is the joint convexity, and the final equality is the unitary invariance of the relative entropy.   This proves the monotonicity of the relative entropy under partial traces.  

The rest is an easy consequence of Theorem~\ref{strep}  It is evident that for $mp \geq q$,
$$D(\Xi_{mp,q}(X)||\Xi_{mp,q}(Y)) = D(X||Y)$$
since the extra zero entries change nothing.   Again using the unitary invariance,
$$D(\ U(\Xi_{mp,q}^\dagger(X)) U^* \ ||\  U(\Xi_{mp,q}^\dagger(Y) )U^*\ ) = D(X||Y)\ .$$
That is, the first two components in the factorization of $\Phi^\dagger$ given by \eqref{stine2} do not have any effect at all on the relative entropy. 
\end{proof}

Thus, to prove that the relative entropy is monotone decreasing under general completely positive trace preserving maps $\Phi^\dagger$, {\em it suffices to prove that it is monotone decreasing under the partial trace}, $\Psi_m^\dagger$, and   any strict inequality in the Data Processing Inequality arises from taking the partial trace; the other two factors in the Stinespring factorization of $\Phi^\dagger$ have no effect on the relative entropy.  This fact is useful in studying cases of equality and stability for the Data Processing Inequality \cite{CV20}. 

We next prove Theorem~\ref{L1M}. The approach is similar, but now we shall need to invoke Remark~\ref{inv}.

\begin{proof}[Proof of Theorem~\ref{L1M}]  As before, we first prove monotonicity under the partial trace. That is, we show that for all $K\in M_q(\C)$, and all $X,Y\in M_{mq}^+(\C)$, and all $0 < t < 1$, 
\begin{equation*}
\tr[\Psi_m(K^*) Y^{1-t} \Psi_m(K) X^t]   \leq  \tr[ K^* \Psi_m^\dagger(Y)^{1-t} K \Psi_m^\dagger(X)^t]\ .
\end{equation*}
The right hand side is equal  to
$$\frac1m \tr[ \Psi_m(K^*) \Psi_m(\Psi_m^\dagger(Y)^{1-t} \Psi_m(K) \Psi_m(\Psi_m^\dagger(X)^{t}]\ ,$$
since inserting $\Psi_m$ in front of each factor inside the trace simple produced $m$ replicas of the original product. Now use the homogeneity to bring the factor of $1/m$ inside and then apply Lemma~\ref{ullm}:

\begin{eqnarray*}
\tr[ K^* \Psi_m^\dagger(Y)^{1-t} K \Psi_m^\dagger(X)^t]  &=& \tr\left[ \Psi_m(K^*) \left(\frac1m\Psi_m(\Psi_m^\dagger(Y)\right) ^{1-t}  \Psi_m(K)\left(\frac1m \Psi_m(\Psi_m^\dagger(X)\right)^{t}\right]\\
&=&   \tr\left[ \Psi_m(K^*) \left(\frac{1}{m^2} \sum_{k=1}^{m^2} U_k^* Y U_k \right) ^{1-t} \Psi_m(K) \left(\frac{1}{m^2} \sum_{k=1}^{m^2} U_k^* XU_k\right)^{t}\right]\\
&\geq&
\frac{1}{m^2} \sum_{k=1}^{m^2} \tr\left[ \Psi_m(K^*) \left( U_k^* Y^{1-t} U_k \right)  \Psi_m(K) \left( U_k^* X^t U_k\right)\right]\\
&=&
\frac{1}{m^2} \sum_{k=1}^{m^2} \tr\left[ \left(U_k\Psi_m(K^*)  U_k^*\right) Y^{1-t} \left(U_k  \Psi_m(K)  U_k\right)X^t \right]\\
&=& \tr[\Psi_m(K^*) Y^{1-t} \Psi_m(K) X^t] \ ,
\end{eqnarray*}
where the second equality is from Lemma~\ref{ullm}, the inequality is the Lieb Concavity Theorem, the third equality is simple regrouping of terms and cyctlicty of the trace, and the final equality is from Remark~\ref{inv}. 

Now consider a general completely positive trace preserving map $\Phi^\dagger$ with the Stinespring factorization $\Phi^\dagger(Y) =   \Psi_m^\dagger( U\Xi_{mp,q}^\dagger(Y) U^* )$.
Then by the above
\begin{eqnarray*}
\tr[ K^* \Phi^\dagger(Y)^{1-t} K \Phi^\dagger(X)^t]  &=& \tr[ K^* \Psi_m^\dagger(U\Xi_{mp,q}^\dagger(Y) U^* )^{1-t} K \Psi_m^\dagger( U\Xi_{mp,q}^\dagger(X) U^* )^t] \\
&\geq&  \tr[ \Psi_m(K^*) ( U\Xi_{mp,q}^\dagger(Y) U^* )^{1-t} \Psi_m(K) ( U\Xi_{mp,q}^\dagger(X) U^* )^t]\\
&=&  \tr[ (U^*\Psi_m(K^*) U)(\Xi_{mp,q}^\dagger(Y) )^{1-t} (U^* \Psi_m(K) U)(\Xi_{mp,q}^\dagger(X) )^t]\\
&=&  \tr[ (U^*\Psi_m(K^*) U)(\Xi_{mp,q}^\dagger(Y^{1-t}) ) (U^* \Psi_m(K)  U)(\Xi_{mp,q}^\dagger(X^t) )]\\
&=&  \tr[\Phi(K^*) Y^{1-t} \Phi(K) X^t]  \ ,
\end{eqnarray*}
where the last equality is from the obvious fact that for an $A,B\in M_{mp}(\C)$ and any $C,D\in M_q(\C)$, 
$\tr[ A \Xi^\dagger_{mp,q}(C) B \Xi^\dagger _{mp,q}(D)]  = \tr[ \Xi_{mp,q}(A) C \Xi_{mp,q}(B) D]$.
\end{proof}

\begin{proof}[Proof of Theorem~\ref{L3M}]
Again, we first prove monotonicity under the partial trace. 
That is, we show that for all $K\in M_{mq}(\C)$, and all $X,Y\in M_{mq}^+(\C)$, and all $0 < t < 1$, 
\begin{equation}\label{lieb23BB} 
 \tr \left[   \int_0^\infty  \Psi_m^\dagger(K^*)\frac{1}{sI + \Psi_m^\dagger(Y)} \Psi_m^\dagger(K) \frac{1}{sI + \Psi^\dagger_m(X)}{\rm d}s \right]   
  \leq   \tr \left[   \int_0^\infty K^*\frac{1}{sI + Y} K \frac{1}{sI + X}{\rm d}s \right]\ .
\end{equation}
We start with the left hand side and note that it is equal to 
\begin{multline*}
\frac1m \tr \left[   \int_0^\infty  \Psi_m(\Psi_m^\dagger(K^*))\Psi_m\left(\frac{1}{sI + \Psi_m^\dagger(Y)}\right)\Psi_m( \Psi_m^\dagger(K) )\Psi_m\left(\frac{1}{sI + \Psi^\dagger_m(X)}\right){\rm d}s \right]  =\\
\tr \left[   \int_0^\infty  \Psi_m\left (\frac1m \Psi_m^\dagger(K^*)\right)\Psi_m\left(\frac {1} {\tfrac{s}{m}I + \tfrac1m \Psi_m^\dagger(Y)}\right)\Psi_m\left( \frac1m\Psi_m^\dagger(K) \right)\Psi_m\left( \frac{1}{\tfrac{s}{m}I + \tfrac1m \Psi^\dagger_m(X)}\right)\frac{{\rm d}s }{m}\right]  \ .
\end{multline*} 
Then since $\Psi_m$ is a $*$-isomorphism,
$$
\Psi_m\left( \frac{1}{\tfrac{s}{m}I + \tfrac1m \Psi_m^\dagger(Y)}\right)  =    \frac{1}{\Psi_m\left(\tfrac{s}{m}I_q + \tfrac1m\Psi_m^\dagger(Y)\right)}   =      \frac{1}{\left(\tfrac{s}{m}I_{mq} + \tfrac1m\Psi_m( \Psi_m^\dagger(Y)\right)} 
$$
where we have written $I_q$ to denote the identity in $M_q(\C)$, and  $I_{mq}$ to denote the identity in $M_{mq}(\C)$.
Thus, the left side of \eqref{lieb23BB} is equal to 
$$
\tr \left[   \int_0^\infty  \left(\frac1m \Psi_m(\Psi_m^\dagger(K^*)\right)
 \frac{1}{\left(t I_{mq} + \tfrac1m\Psi_m( \Psi_m^\dagger(Y)\right)} 
  \left(\frac1m \Psi_m(\Psi_m^\dagger(K)\right)
  \frac{1}{\left(t I_{mq} + \tfrac1m\Psi_m( \Psi_m^\dagger(X)\right)}{\rm d}t \right] 
$$
We now apply Lemma~\ref{ullm} together with Theorem~\ref{L3}    as before, and then use  the obvious unitary invariance to bound this above by 
$$
 \tr \left[   \int_0^\infty K^*\frac{1}{tI + Y} K \frac{1}{tI + X}{\rm d}t \right]\ ,
$$
and this proves \eqref{lieb23BB}. To obtain the general case, we merely need to observe that for each $t>0$,
\begin{multline*}
 \tr \left[  U\Xi_{mp,q}^\dagger(K^*) U^* \frac{1}{tI + U\Xi_{mp,q}^\dagger(Y) U^*} U\Xi_{mp,q}^\dagger(K) U^*\frac{1}{tI + U\Xi_{mp,q}^\dagger(X) U^*} \right] =\\
\tr \left[   K^*\frac{1}{tI + Y} K \frac{1}{tI + X} \right]
\end{multline*}
and then apply Theorem~\ref{strep}.
\end{proof}

The remaining two cases involve negative powers,   and as we have observed, completely positive trace preserving maps can take positive definite matrices to non-invertible positive semidefinite matrices. A good example is $\Xi_{mp,q}$ for $mp > q$. However, the partial trace of a positive definite matrix is always positive definite, and so no complication arises in the main part of the argument -- proving monotonicity under the partial trace. One does that as before with positive definite $X$ and $Y$. 

Associated to any completely positive trace preserving map $\Phi^\dagger:M_m(\C) \to M_n(\C)$ and any $0 < \epsilon < 1$, there is the completely positive trace preserving map
$$\Phi_\epsilon^\dagger(A) = (1- \epsilon)\Phi^\dagger(A) + \epsilon  \frac1n \tr[A] I\,
$$
for which the right hand side is invertible for all non-zero  $A\in M_m^+(\C)$, and moreover if $\Phi^\dagger(A)$ is invertible, then $\lim_{\epsilon\downarrow 0} (\Phi_\epsilon^\dagger(A))^{-1} = (\Phi^\dagger(A))^{-1}$.  Finally, if  $\Phi^\dagger(Y) =   \Psi_m^\dagger( U\Xi_{mp,q}^\dagger(Y) U^* )$ is the canonical Stinespring factorization of $\Phi^\dagger$ given by Theorem~\ref{strep},  if we replace $\Xi_{mp,q}^\dagger$ by $(\Xi_{mp,q}^\dagger)_\epsilon$ a simple computation shows that
\begin{equation}\label{srepmod}
\Phi^\dagger_\epsilon(A) =  \Psi_m^\dagger (U (\Xi_{mp,q}^\dagger)_\epsilon(A)  U^*)\ .
\end{equation}

\begin{proof}[Proof of Theorem~\ref{L2M} ]  We first prove that there is monotonicity under partial traces; i.e.,
 for  all $X,Y\in M_{mq}^{++}(\C)$, all $K \in M_{mq}(\C)$ and all $0 < t < 1$,
\begin{equation}\label{L223}
  \tr[ \Psi_m^\dagger(K^*) \Psi_m^\dagger(Y)^{t-1}  \Psi_m^\dagger(K) \Psi_m^\dagger(X)^{-t}]  \leq   \tr[ K^*Y^{t-1}  KX^{-t}]
\end{equation}
This goes as before; we rewrite the left hand side in terms of a trace over $m\times m$ diagonal entries, obtaining that it is equal to
\begin{multline*}
\frac1m  \tr[ \Psi_m(\Psi_m^\dagger(K^*)) \Psi_m(\Psi_m^\dagger(Y)^{t-1})  \Psi_m(\Psi_m^\dagger(K)) \Psi_m(\Psi_m^\dagger(X)^{-t})]  =\\
\tr[(\tfrac1m  \Psi_m(\Psi_m^\dagger(K^*))) (\tfrac1m \Psi_m(\Psi_m^\dagger(Y))^{t-1})  (\tfrac1m \Psi_m(\Psi_m^\dagger(K))) (\tfrac1m \Psi_m(\Psi_m^\dagger(X))^{-t})] 
\end{multline*}
where we have used the homogeneity. Now use Lemma~\ref{ullm}, Theorem~\ref{L2} and the unitary invariance to conclude \eqref{L223}.

 Now let
$\Phi^\dagger(Y) =   \Psi_m^\dagger( U\Xi_{mp,q}^\dagger(Y) U^* )$ be the canonical Stinespring factorization of $\Phi^\dagger$, and then for $A,B\in M_q^{++}(\C)$, $C\in M_q(\C)$, 
$0 < \epsilon < 1$, define
$$
X := U( \Xi_{mp,q}^\dagger)_\epsilon(A) U^* \ ,\quad  Y:= U (\Xi_{mp,q}^\dagger)_\epsilon(B) U^*   \quad{\rm and}\quad   K= U( \Xi_{mp,q}^\dagger)_\epsilon(C) U^*\ .
$$
Inserting these into \eqref{L223} yields, using \eqref{srepmod} on the left and unitary invariance on the right,
\begin{multline*}
  \tr[ \Phi_\epsilon^\dagger(C^*) \Phi_\epsilon^\dagger(B)^{t-1}  \Phi_\epsilon^\dagger(C) \Phi_\epsilon ^\dagger(A)^{-t}]  \leq\\   
  \tr[ ( \Xi_{mp,q}^\dagger)_\epsilon(C^*)( \Xi_{mp,q}^\dagger)_\epsilon(B)^{t-1} \Xi_{mp,q}^\dagger)_\epsilon(C^*)  ( \Xi_{mp,q}^\dagger)_\epsilon(A) ^{-t}]\ .
\end{multline*}
Now take $\epsilon \downarrow 0$. 
\end{proof}

\begin{proof}[Proof of Theorem~\ref{ACTM}]   The proof of the monotonicity under partial traces proceeds in exactly the same way as in the proof of Theorem~\ref{L1M} above, but with convexity replacing concavity. Beyond this point one cannot proceed; the  device of replacing  $\Xi_{mp,q}^\dagger$ by $ (\Xi_{mp,q}^\dagger)_\epsilon$ does not help this time because of the way the terms involving $\Xi_{mp,q}$ are separated, and the example given above shows that no other device will help. 
\end{proof}

\section{Strong subadditivity of the quantum entropy}

The first application of the inequalities proved in \cite{L73} was the proof of the {\em strong subadditivity of the quantum entropy} (SSA) conjecture. To explain,
recall that if $\rho$ is a density matrix on a Hilbert space,  its von Neumann entropy, $S(\rho)$, is defined by
$$
S(\rho) = -\tr[\rho \log \rho]\ .
$$
Now consider a density matrix $\rho_{12}$ on a bipartite Hilbert space $\cH_1\otimes \cH_2$.  Let $\rho_1$ be the partial trace of $\rho_{12}$ over $\cH_2$; i.e., $\rho_1 = \tr_2[\rho_{12}]$, and define $\rho_2$ in the analogous manner. Then $\rho_1\otimes \rho_2$ is a density matrix on $\cH_1\otimes\cH_2$, and
\begin{eqnarray*}
0 \leq D(\rho_{12}||\rho_1\otimes \rho_2) &=& \tr[ \rho_{12}(\log \rho_{12} - \log{\rho_1}\otimes I - I\otimes \log \rho_2)] \\
&=& -S(\rho_{12}) + S(\rho_1) + S(\rho_2)\ ,
\end{eqnarray*}
where the initial inequality is \eqref{klein}.
Going forward,  it will be useful to simplify our notation and write $S_{12} = S(\rho_{12})$, $S_1 = S(\rho_1)$, etc., and to simply write  $\log \rho_1$ in place of $\log{\rho_1}\otimes I$, etc.  Then from the calculation just made
\begin{equation*}
S_{12} \leq S_1 + S_2  \quad{\rm and}\quad S_{12} = S_1+S_2 \ \iff \ \rho_{12} = \rho_1\otimes \rho_2\ .
\end{equation*}
This relatively elementary, but physically important, inequality is known as the {\em subadditivity of the quantum entropy}.

SSA concerns a density matrix $\rho_{123}$ on a tripartite Hilbert space $\cH_1\otimes \cH_2\otimes \cH_3$.  The 1968 conjecture of  Lanford and Robinson   \cite{LR68} was that 
\begin{equation}\label{SSA}
S_{12}  + S_{23} \geq S_{123} + S_2\ .
\end{equation}
If $\cH_2$ is one dimensional, then \eqref{SSA} reduces to $S_{1} + S_3 \geq S_{13}$; i.e., to subadditivity, and this justifies the name strong subadditivity. 

The name SSA  is generally used in the statistical mechanics literature for what is known in classical information theory as {\em positivity of the conditional mutual information}, 
as discussed at the end of this section.  The term SSA was introduced in the context of classical statistical mechanics by Robinson and Ruelle \cite[Proposition 1]{RR} who gave a proof in this context, 
and applied it to prove of the existence of the thermodynamic limit in classical statistical mechanics.  Then in 1968, Lanford and Robinson attempted to apply the methods 
of \cite{RR} in the quantum setting, but fell short as they could not prove quantum SSA. They explicitly stated  the conjecture that \eqref{SSA} was valid in the quantum setting \cite[p. 1125]{LR68}, 
writing, in reference to classical SSA, that: ``One could believe, and even support one's belief by heuristic physical arguments that the same condition holds for the quantum entropy".
For an overview of the physical context, see the review article of Wehrl \cite{W78}. 

The classical analog is  elementary  to prove, and this is done below, but the {\em usual} classical proof uses conditional probabilities.  Recall that if $\rho(x,y)$ is a joint probability 
distribution for two random variables, and $\rho(x)$ is the marginal distribution for one of them, the corresponding conditional probability density $\rho(x|y)$ is $\rho(x,y)/\rho(y)$.  Unfortunately in the quantum setting there are many ways one might try divide $\rho_{12}$ by $\rho_2$, but none of them yields a 
satisfactory notion of conditional quantum probability, and none of them provide a basis for adapting the easy proof of classical SSA  to the quantum case. 

However, once the results from \cite{L73} were available, Lieb and Ruskai \cite{LR}  proved SSA, also in 1973. Moreover, they gave several equivalent reformulations of SSA that have turned out to be very important in their own right.   One of these involves the notion of {\em conditional entropy}:
 \begin{defi}\label{cedef} Given a density matrix $\rho_{12}$ on $\cH_1\otimes \cH_2$, considered as a state on a bipartite system, the {\em conditional entropy of system 1 with respect to system 2} is defined to be $S_{12} - S_1$. 
 \end{defi}
 
 At the end of this section, we explain the connection with the concept of conditional entropy in classical probability theory, 
but since the whole notion of conditional probability is problematic in quantum mechanics, the origins of the name are largely irrelevant for present purposes.

 \begin{thm}[Lieb Ruskai 1973, SSA and two equivalent formulations]\label{SSAmain}  The following statements are all true, and equivalent in that once any one of them is proven, simple arguments yield the other two:
 
 \medskip
 \noindent{\it (1)} For all density matrices $\rho_{123}$ on $\cH_1\otimes \cH_2\otimes \cH_3$, \eqref{SSA} is satisfied.
 
 \medskip
 \noindent{\it (2)}  The map
$\rho_{12} \mapsto S_2 - S_{12}$
is convex and homogeneous of degree one  on the set of density matrices $\rho_{12}$ on $\cH_1\otimes \cH_2$. 

\medskip
 \noindent{\it (3)} The relative entropy functional $(\rho,\sigma) \mapsto D(\rho ||\sigma) $  is montone under partial traces; i.e., for all density matrices $\rho_{1,2},\sigma_{1,2}$ on $\cH_1\otimes \cH_2$, $D(\rho_1||\sigma_1) \leq D(\rho_{12}||\sigma_{12})$. 
 \end{thm}

As in \cite{LR}, we use an elementary Lemma from Lieb's 1973 paper, \cite[Lemma 5]{L73}, whose simple statement and proof we recall. (It shall be applied several times in what follows.)

\begin{lm}[Lieb 1973]\label{lieblem} Let $\mathcal{C}$ be a convex cone in a vector space, and let $F:\mathcal{C}\to \R$ be  homogeneous of degree one. Define
$$
G(x,y) := \lim_{t \downarrow 0} \frac{F(x+ty) - F(x)}{t}\ .
$$
If $F$ is convex, then for all $x,y\in \mathcal{C}$, $G(x,y) \leq F(y)$.  Conversely, if $F$ is continuously differentiable and  $G(x,y) \leq F(y)$ for all $x,y\in \mathcal{C}$,
then $F$ is convex.
\end{lm}

\begin{proof} Assume first that $F$ is convex, and compute
\begin{eqnarray*}
F(x + ty) &=& (1+t) F\left( \frac{1}{1+t}x + \frac{t}{1+t}y\right) \leq (1+t)\left( \frac{1}{1+t}F(x) + \frac{t}{1+t}F(y)\right)\\
&=& F(x) + tF(y)\ .
\end{eqnarray*}
Next assume $F$ is continuously differentiable and  $G(x,y) \leq F(y)$ for all $x,y\in \mathcal{C}$. For $0 < \lambda < 1$, $x,y\in \mathcal{C}$,
let $z:= \lambda x + (1-\lambda)y$. Then
$$F(z) - \lambda F(x) = F(\lambda x + (1-\lambda)y) - F(\lambda x) = \int_0^{1-\lambda}G(\lambda x + sy,y){\rm d}s \leq (1-\lambda)F(y)\ .$$
\end{proof}

 \begin{proof}[Proof of Theorem~\ref{SSAmain}]    First we show that {\it(3)} $\Rightarrow$ {\it (1)}:
Let $\tau_{3}$ denote the density matrix on 
$\cH_3$ that is a multiple of the identity, i.e. $\tau_{3} = (d_3)^{-1}I_{3}$ where $d_3$ is the dimension of $\cH_3$.  
 Since $D(X||Y) = \tr[X(\log X - \log Y)]$, 
$$
D(\rho_{123} || \rho_{12}\otimes \tau_3) = -S_{123} + S_{12} + \log d_3  \quad{\rm and}\quad 
D(\rho_{23}||\rho_2\otimes\tau_3) = -S_{23} + S_2 +\log d_3\ .
$$
Therefore
\begin{equation}\label{SSA2}
S_{12}  + S_{23} - S_{123} - S_2   = D(\rho_{123} || \rho_{12}\otimes \tau_3)  -  D(\rho_{23}||\rho_2\otimes\tau_3)\ .
\end{equation}
Let $\Phi^\dagger$ denote the partial trace over $\cH_1$ so that $\Phi^\dagger$
 is completely positive and trace preserving. Then since
$\Phi^\dagger(\rho_{123}) = \rho_{23}$ and $\Phi^\dagger(\rho_{12}\otimes \tau_3) = \rho_{2}\otimes \tau_3$,
$D(\rho_{123} || \rho_{12}\otimes \tau_3)  -  D(\rho_{23}||\rho_2\otimes\tau_3)  \geq 0$ as a consequence of the monotonicity of the relative entropy under partial traces, and then by \eqref{SSA2}, this implies \eqref{SSA}. 

We next prove that  {\it(1)} $\Rightarrow$ {\it (2)}.
 For two density matrices $\rho_{12}$ and $\sigma_{12}$ on 
$\cH_1\otimes \cH_2$, take $\cH_3$ to be the $\C^2$ with the usual inner product and define the tripartite state
$
\frac12 \left[\begin{array}{cc} \rho_{12} & 0\\ 0 & \sigma_{12}\end{array}\right]$.
Then
\begin{equation*}
S_{123} = \log 2 -\frac{1}{2}\tr[\rho_{12} \log \rho_{12}]  -\frac{1}{2}\tr[\sigma_{12} \log \sigma_{12}]\ ,
\end{equation*}
\begin{equation*}
S_{23} =\log 2 -\frac{1}{2}\tr[\rho_{2} \log \rho_{2}] -\frac{1}{2}\tr[\sigma_{2} \log \sigma_{2}] \ ,\quad 
S_{12} = -\tr\left[\frac{\rho_{12} +\sigma_{12}}{2}  \log\left(   \frac{\rho_{12} +\sigma_{12}}{2}  \right)\right]    \ ,
\end{equation*}
and 
  \begin{equation*}
S_{2} = -\tr\left[\frac{\rho_{2} +\sigma_{2}}{2}  \log\left(   \frac{\rho_{2} +\sigma_{2}}{2}  \right)\right]  \ .
\end{equation*}
Now \eqref{SSA} yields
$$
S\left(\frac{ \rho_{2} +\sigma_{2}}{2}\right)  -S\left(\frac{ \rho_{12} +\sigma_{12}}{2}\right)   \leq 
\frac12\left(- S(\rho_2) - S(\rho_{12} \right) + \frac12\left(S(\sigma_2)- S(\sigma_{12} \right)\ .
$$
This proves the convexity, and  the homogeneity is obvious. 

We next prove that {\it (2)} $\Rightarrow$ {\it(3)}.
Fix two density matrices $\rho_{12}$ and $\sigma_{12}$ on $\cH_1\otimes \cH_2$, and compute
$$
\lim_{t \downarrow 0} \frac1t (S(\sigma_{12} + t \rho_{12}) -S(\sigma_{12}))  = -1 -\tr[\rho_{12} \log \sigma_{12}]\ .
$$
By Lemma~\ref{lieblem},
$
 (-1 -\tr[\rho_{2} \log \sigma_{2}])   - \left(-1 -\tr[\rho_{12} \log \sigma_{12}]\right)   \leq  -\tr[\rho_2\log \rho_2]  +\tr[\rho_{12}\log \rho_{12}]  
$.
Rearranging terms,
$D(\rho_{12}||\sigma_{12}) \geq D(\rho_2||\sigma_2)$,
which is the monotonicity of the relative entropy under partial traces. 

Finally, since the monotonicity of the relative entropy under partial traces is a special case of the DPI (from which the general case can be deduced), all three of the equivalent statements are valid. 
\end{proof}

\begin{remark} Because the joint convexity of the relative entropy, and hence the DPI,  is a simple consequence of the Lieb Concavity Theorem, one can view all of these equivalent statements as fairly direct consequences of Theorem~\ref{L1}. Furthermore, if one applies the converse of Lemma~\ref{lieblem}, one easily obtains a direct proof that {\it (3)} implies {\it (2)}.  Likewise, another application of Uhlmann's Lemma shows that ${\it (2)}$ can be rephrased as saying that the conditional entropy is monotone under partial traces, and then it is easy to see that this yields SSA. 

The proof that we have presented here is somewhat different from the proof in \cite{LR}. In \cite{LR} it is proved that {\it (1)} and {\it (2)} are equivalent, both of which are also given direct proofs, and the paper ends by showing in the final paragraph that {\it (2)} implies {\it (3)} without mentioning the term  ``relative entropy'' by name, and without noting the equivalence of {\it (2)} and {\it (3)} which follows directly from the second half of Lemma~\ref{lieblem}, which was used for such purposes twice in \cite{L73}. 

As explained in Section 3, the Data Processing Inequality (DPI)  \eqref{lieb8} is easily seen to be equivalent to the monotonicity of the relative entropy under partial traces,  and hence the DPI is simply an equivalent formulation of SSA.   It is in this form that SSA finds its widest use today, in quantum information theory, although as explained in Section 5, the original formulation has important applications there as well. 
\end{remark} 

The conjecture of Lanford and Robinson attracted much attention, and many people had worked to prove it during the half decade it was open. Among the first  published papers resulting from efforts in this direction  is a paper of Araki and Lieb \cite{AL}.  We quote Theorem 1 of that paper:

\begin{thm}[Araki-Lieb 1970]\label{ALSSA}  Let $\rho_{123}$ be a tripartite density matrix.  Then 
\begin{equation}\label{SSAW}
S_{12}  + S_{23} \geq S_{123} - \log \tr[\rho_2^2]\ .
\end{equation}
If $\rho_2$ commutes with $\rho_{12}$ or $\rho_{23}$, then \eqref{SSA} is valid.
\end{thm}

By Jensen's Inequality, and the concavity of the logarithm, if $\{\lambda_1,\dots,\lambda_n\}$ are the eigenvalues of $\rho_2$, repeated according to multiplicity, 
$$
\log \tr[\rho_2^2]  = \log\left(\sum_{j=1}^n \lambda_j \lambda_j\right) \geq \sum_{j=1}^n \lambda_j \log \lambda_j\ .
$$
Therefore,
$- \log \tr[\rho_2^2]  \leq S(\rho_2)$, and hence \eqref{SSAW} is weaker than \eqref{SSA}.

Nonetheless,  Theorem~\ref{ALSSA}, together with the Araki-Lieb Triangle Inequality, also proved in \cite{AL} and discussed in the next section, 
were sufficient \cite{AL} to solve the quantum thermodynamic limit problem that had motivated the conjecture of Lanford and Robinson in the first place. 
Despite this, interest in proving the original conjecture remained strong, and the progress of Araki and Lieb influenced future work. 
Hence it is worthwhile to recall the simple proof of Theorem~\ref{ALSSA}.   The proof uses the well-known Golden-Thompson inequality  \cite{G65,T65} which states that
for all self-adjoint $H,K\in M_n(\C)$,
\begin{equation}\label{GT}
\tr[e^{H+K}] \leq \tr[e^H e^K]\ ,
\end{equation}
and the Peierls-Bogoliubov Inequality which states that  for self-adjoint $H$ and $K$ with $\tr[e^{H}] =1$, 
\begin{equation}\label{PB}
\tr[Ke^H] \leq \log \tr[e^{H+K}]  \ .
\end{equation}
An elementary proof can be found in many places; e.g., \cite[Appendix A]{CL19}.  Note that if  $\rho(x)$ is a classical probability density on a   measure space $(X,\mu)$, and $f(x)$ is any bounded, real valued  function on $X$, 
\begin{equation}\label{BogCl}
\int_X f(x) \rho(x){\rm d}\mu \leq \log\left( \int_X e^{f(x)}\rho(x){\rm d\mu}\right)  = \log\left( \int_X e^{f(x) + \log \rho(x)}{\rm d\mu}\right)\ ,
\end{equation}
by Jensen's inequality for the exponential function.  Replacing integrals by traces, $f$ by $K$ and $\rho$ by $e^{H}$, the left side of \eqref{BogCl} becomes 
the left side of \eqref{PB}.  The two terms on the right hand side of \eqref{BogCl} become $\log(\tr[e^Ke^H])$ and $\log(\tr[e^{H+K}])$, which are no longer equal. By the Golden-Thompson inequality, $\log(\tr[e^{H+K}]) \leq \log(\tr[e^Ke^H])$, and the utility of the Peierls-Bogoliubov Inequality lies in the fact that it is the {\em smaller}  of these two quantities that appears on the right side of \eqref{PB}.  This is essential in the application to follow.

\begin{proof}[Proof of Theorem~\ref{ALSSA}]
Assume that $\rho_{123}$ is positive definite, and define
$$
\Delta := S_{123} - S_{12} - S_{23} = \tr[\rho_{123}( -\log \rho_{123} +\log\rho_{12} + \log \rho_{23})]\ .
$$
Now apply \eqref{PB}  taking $H= \log \rho_{123}$ and $K = -\log \rho_{123} +\log\rho_{12} + \log \rho_{23}$
\begin{multline*}
e^{\Delta} \leq \tr[ \exp(\log \rho_{123} -\log \rho_{123} +\log\rho_{12} + \log \rho_{23})]  = \\\tr[ \exp(\log\rho_{12} + \log \rho_{23}))] \leq \tr[\rho_{12}\rho_{23}] = \tr[\rho_2^2]
\end{multline*}
where the second inequality is \eqref{GT}. This proves \eqref{SSAW}.

Now suppose $\rho_2$ commutes with $\rho_{23}$,  and make a different definition of $\Delta$:
$$
\Delta := S_{123} +S_2  - S_{12} - S_{23} = \tr[\rho_{123}( -\log \rho_{123} - \log \rho_2 +\log\rho_{12} + \log \rho_{23})]\ .
$$
Using the Peierls-Bogoliubov Inequality as before, 
\begin{equation}\label{ALKey}
e^{\Delta} \leq \tr[ \exp( - \log \rho_2 +\log \rho_{12} + \log \rho_{23})]\ .
\end{equation}
Then again by the Golden-Thompson Inequality, and then the fact that $\rho_2$ commutes with $\rho_{23}$,
\begin{multline*}
 e^\Delta \leq \tr[ \exp(- \log \rho_2  + \rho_{12} + \log \rho_{23}))] \leq \tr[\rho_{12} \exp(-\log \rho_2  + \log \rho_{23})] =\\ \tr[ \rho_{12}\rho_2^{-1}\rho_{23}]  = 1\ .
\end{multline*}
Therefore, $\Delta \leq 0$, and hence \eqref{SSA} is valid in this case. 
\end{proof}

This result was very influential in the community of people working on proving the conjecture of Lanford and Ruelle. It put the focus on the need for a stronger form of the 
Golden-Thompson inequality that would provide an upper bound on $\tr[e^{H+K+L}]$ for arbitrary self-adjoint matrices $H$, $K$ and $L$ without assuming that any of them 
commute. Uhlmann in particular took up this line of investigation, and he proved \cite[Satz 8.2]{Uh73}  a mild generalization of part of Theorem~\ref{ALSSA}, namely that 
\eqref{SSA} is valid when any two of the matrices $\rho_{12}$, $\rho_{23}$ and $\rho_2$ commute. 

The fifth theorem in Lieb's paper \cite{L73} is a generalization of the Golden-Thompson Inequality to three self-adjoint matrices. The naive generalization might be (depending on how naive one might be), 
$\tr[e^{H+K+L}] \leq \tr[e^H e^K e^L]$, but this would be complete nonsense;  the right side is in general a complex number.   Lieb's {\em Triple Matrix Inequality}, the fifth theorem in \cite{L73},  states that 
$$
\tr[e^{H+K+L}]   \leq \tr[e^{H}T_{e^{-K}}(e^L)]] \ ,
$$
where
$$T_A(B) = \int_0^\infty \frac{1}{s+A} B \frac{1}{s+A} {\rm d}s\ .$$
Notice that if $A$ and $B$ commute, we have $T_A(B) = A^{-1}B$, and hence if $K$ and $L$ commute, $ e^{H}T_{e^{-K}}(e^L) = e^{H}e^{K+L}$. A proof of this theorem is given in Section 7; see Theorem~\ref{TripleMatrix}.

Armed with this, return to \eqref{ALKey}, and observe that, using Lieb's Triple Matrix Inequality in the second line, and taking the trace over $\cH_3$ in the next step, and then using cyclicity of the trace, 
\begin{eqnarray*}
e^{\Delta} &\leq& \tr[ \exp( - \log \rho_2 + \log \rho_{12} + \log \rho_{23})] \\
&\leq& \int_0^\infty \tr\left[ \frac{1}{s+\rho_2} \rho_{12}  \frac{1}{s+\rho_2} \rho_{23} \right] {\rm d}s\\
&=& \int_0^\infty \tr\left[ \frac{1}{s+\rho_2} \rho_{12}  \frac{1}{s+\rho_2} \rho_{2} \right] {\rm d}s\\
&=& \int_0^\infty \tr\left[ \rho_{12} \rho_{2} \frac{1}{(s+\rho_2)^2}  \right] {\rm d}s = \tr[\rho_{12}]  = 1\ .
\end{eqnarray*}
This proves SSA following the line of the investigation started in \cite{AL},  more directly than in \cite{LR}, despite the fact that Lieb wrote in the abstract  of \cite{L73} that the results were relevant to the proof of SSA,  and had, as he has told me, SSA in mind when he proved the Triple Matrix Inequality.   Only very recently have generalizations of the Golden-Thompson inequality to more than three matrices been proved. See \cite{SBT} for these theorems and their application to questions in quantum information theory.

We close this section by discussing the classical analogs of the inequalities  we have been considering.
Given two finite sets $\mathcal{X}$ and $\mathcal{Y}$, and a  probability density $\rho(x,y)$ on $\mathcal{X}\times \mathcal{Y}$, thought of as specifying the joint distributions of two discrete random variables $X$ and $Y$,  the {\em (classical discrete) conditional  entropy} $H(X|Y)$ of $X$ with respect to $Y$ is defined by
\begin{equation}\label{cond12}
H(X|Y) = -\sum_{(x,y)\in \mathcal{X}\times \mathcal{Y}} \rho(x,y)\log\left(\frac{\rho(x,y)}{\rho(y)}\right) \ ,
\end{equation}
where
$$\rho(y) := \sum_{x\in \mathcal{X} }\rho(x,y)$$ 
is the {\em marginal distribution} of $Y$.  
If we denote
\begin{equation}\label{clasmarg}
S_{12} = -\sum_{(x,y)\in \mathcal{X}\times \mathcal{Y}}  \rho(x,y)\log \rho(x,y)    \qquad{\rm and}\qquad  S_2 := - \sum_{y\in \mathcal{Y} } \rho(y)\log\rho(y)\ ,
\end{equation}
we then find that
\begin{equation}\label{cone}
H(X|Y) = S_{12} - S_2\ .
\end{equation}
Note that
\begin{equation}\label{clconpos}
H(X|Y) = -\sum_{y\in \mathcal{Y} } \left(\sum_{x\in \mathcal{X} }\left(\frac{\rho(x,y)}{\rho(y)}\right) \log\left(\frac{\rho(x,y)}{\rho(y)}\right) \right) \rho(y)  \geq 0
\end{equation}
since for all $x,y$, $0 \leq \rho(x,y)/\rho(y) \leq 1$. 
Thus, in the  setting of classical discrete probability, it is always the case that $S_2 \leq S_{12}$, and it is easy to see that there is equality if and only if $\rho(x,y)/\rho(y) \in \{0,1\}$ for all $x,y$, and this is the case if and only if the random variable $X$ is completely determined by the random variable $Y$; i.e.,  for some $f:\mathcal{Y}\to \mathcal{X}$, $X = f(Y)$.

 While there simply is not a good way to define an analog of the conditional density $\rho(x,y)/\rho(y)$  out of $\rho_{12}$ and $\rho_2$ in the quantum case, 
 and hence no way to use an analog of \eqref{cond12} to define conditional entropy in the quantum case,  the right hand side of \eqref{cone} makes perfect sense in the quantum setting, and leads to  Definition~\ref{cedef}.

One might hope that the inequality $S_{12} - S_1>0$; i.e., \eqref{cone} and \eqref{clconpos} in the classical discrete case,  would extend to the quantum case, {\em but it does not}. In fact, let $\cH$ be an $n$-dimensional Hilbert space, and let $\{u_1,\dots u_n\}$ be an orthonormal basis for $\cH$. Define a state $\psi$ on $\cH\otimes \cH$ by 
$$
\psi := \frac{1}{\sqrt{n}}\sum_{j=1}^n u_j\otimes u_j\ ,
$$
and then define $\rho_{12} :=  |\psi \rangle \langle \psi|$.
Since $\rho_{12}$ is a rank one projection; i.e., a {\em pure state}, $S(\rho_{12}) = 0$. However, simple computations show that
$\rho_1 = \rho_2 = \frac1n I$, and hence $S_1 = S_2 = \log n$ which is easily seen to be the maximum value of the entropy of any density matrix on $\cH$.


%
%
Given  three finite sets $\mathcal{X}$, $\mathcal{Y}$ and  $\mathcal{Z}$  and a tripartite probability density on $\mathcal{X} \times \mathcal{Y}  \times \mathcal{Z}$  form the 
the conditional joint and marginal distributions  of $X$ and $Y$, given $Z = z$:
\begin{equation}\label{condition}
\rho(x,y|z) := \frac{\rho(x,y,z)}{\rho(z)} \, \qquad   \rho(x|z) := \frac{\rho(x,z)}{\rho(z)}\quad{\rm and}\quad   \rho(y|z) := \frac{\rho(y,z)}{\rho(z)}
\end{equation}
By the convexity of $t\mapsto t \log t$ and Jensen's inequality, for each $z$, 
$$
\sum_{(x,y)\in \mathcal{X}\times \mathcal{Y}}   \rho(x|z)   \rho(y|z) \left(\frac{\rho(x,y|z)}{\rho(x|z)\rho(y|z)}\right) 
\log \left(\frac{\rho(x,y|z)}{\rho(x|z)\rho(y|z)}\right) \geq 0\ ,
$$
and therefore the {\em conditional mutual information of $X$ and $Y$ given} $Z$, $I(X,Y|Z)$, defined by 
$$
I(X,Y|Z) = \sum_{z\in \mathcal{Z}}\rho(z)\left(\sum_{(x,y)\in \mathcal{X}\times \mathcal{Y}}   \rho(x|z)   \rho(y|z) \left(\frac{\rho(x,y|z)}{\rho(x|z)\rho(y|z)}\right) 
\log \left(\frac{\rho(x,y|z)}{\rho(x|z)\rho(y|z)}\right)\right)
$$
satisfies $I(X,Y|Z) \geq 0$. A simple calculation shows 
$$
I(X,Y|Z)  = S_{13}+S_{23} - S_{123} - S_3\ ,
$$
and this proves $S_{13}+S_{23} - S_{123} - S_3> 0$, which is SSA  in the classical setting. Specializing to the case in which $\mathcal{Z}$ consists of a single point,
$S_3= 0$, $S_{13} = S_1$, $S_{23} = S_s$ and $S_{123} = S_{12}$, and we obtain the classical subadditivity of the entropy, 
$S_{12} \leq S_1+ S_2$  which obviously can be proved directly using Jensen's Inequality as above.   We have just proved subadditivity and 
strong subadditivity in the  discrete probability setting, but unlike the positivity of conditional entropy \eqref{clconpos}, which is not true for continuous 
random variables, both subadditivity and strong subadditivity are valid in the classical case for both discrete and continuous random variables.  
Strong subadditivity is universal, holding in all contexts, unlike the positivity of conditional probability, which is valid only in the classical discrete setting. 
For more information, see \cite{L75}.

One cannot even begin to adapt the above classical proof to the quantum setting because there is no reasonable way to form 
analogs of the conditional densities \eqref{condition} out of the marginals of a tripartite density matrix $\rho_{123}$. There is, however, 
another proof \cite{CL99} of classical SSA that does not make any reference to conditonal probability, and which does extend to the 
quantum case, as we show in Section 7.

\section{The Araki-Lieb inequality}

This section  recalls some results from \cite{AL}  that were used together with the ``good enough'' SSA, Theorem~\ref{ALSSA} to solve the thermodynamic limit problem that had motivated the conjecture of Lanford and Ruelle. 

One might think that this would have closed the subject, but many people realized that SSA was interesting far beyond its original motivation. 
Both SSA and the inequalities discussed here have indeed proven to be fundamental in quantum information theory. 
It even appears in black hole physics \cite{HT07}.

Let $\rho_{12}$ be a density matrix on the tensor product of finite dimensional Hilbert spaces, $\cH_1\otimes \cH_2$. Then using the natural orthonormal basis of  $\cH_1\otimes \cH_2$ induced by orthonormal bases of $\cH_1$ and $\cH_2$, we may write $\rho_{12}$ as a matrix 
with entries $[\rho_{12}]_{(i,k),(j,\ell)}$. Then with $d_j := {\rm dim}(\cH_j)$, $j=1,2$, 
\begin{equation}\label{AL1}
[\rho_{1}]_{i,j} = \sum_{k=1}^{d_2} [\rho_{12}]_{(i,k),(j,k)} \quad{\rm and}\qquad 
[\rho_{2}]_{k,\ell} = \sum_{i=1}^{d_1} [\rho_{12}]_{(i,k),(i,\ell)}\ .
\end{equation}

Suppose that $\rho_{12}$ is a pure state; i.e., for some unit vector $\psi\in \cH_1\otimes \cH_2$, $\rho_{12} = |\psi\rangle\langle \psi |$ so that 
\begin{equation}\label{AL2}
[\rho_{12}]_{(i,k),(j,\ell)} = \overline{\psi_{i,k}} \psi_{j,\ell}\ .
\end{equation} 
Define a $d_1\times d_2$ matrix $K$ by $K_{j,\ell} = \psi_{j,\ell}$. Then from \eqref{AL1} and \eqref{AL2},
\begin{equation}\label{AL3}
\rho_{1} = KK^* \qquad{\rm and}\qquad \rho_2 = K^*K\ .
\end{equation}
Since $K^*K$ and $KK^*$ have the same non-zero spectrum with the same multiplicities for non-zero eigenvalues, it follows that when $\rho_{12}$ is a pure state, then $S_1 = S_2$.   It is also evident that $S_{12} = 0$ if and only if $\rho_{12}$ is a pure state.  Thus
\begin{equation}\label{AL0}
S_{12} = 0  \ \Rightarrow\ S_1 = S_2   \ . 
\end{equation}
This discussion in this paragraph summarized Lemma 3 of \cite{AL} and its short proof.

The Araki-Lieb Theorem allows one to conclude that $S_1$ is close to $S_2$ if $S_{12}$ is small:

\begin{thm}[Araki-Lieb Triangle Inequality]\label{ALthm}  Let $\rho_{12}$ be any density matrix on the tensor product of finite dimensional Hilbert spaces, $\cH_1\otimes \cH_2$.
Then
\begin{equation}\label{AL4}
|S_1 - S_2| \leq S_{12}\ ;
\end{equation}
\end{thm}

The proof uses a very simple construction known as {\em purification}, which has many other uses, some of which will be discussed later  in this paper.  The following is Lemma 4 from \cite{AL}.

\begin{lm}[Purification Lemma]\label{purlm}  Let $\rho$ be any density matrix on a finite dimensional Hilbert space $\cH$. Then there is a pure state density matrix $|\psi \rangle\langle \psi|$ on $\cH \otimes \cH$ such that
\begin{equation}\label{AL5}
\rho = \tr_1[|\psi \rangle\langle \psi|]= \tr_2[|\psi \rangle\langle \psi|]
\end{equation}
where $\tr_1$ and $\tr_2$ denote the partial traces over $\cH_1$ and $\cH_2$ respectively. 
\end{lm}

\begin{proof} Define $K = \rho^{1/2}$ and using any orthonormal basis $\{u_1,\dots,u_d\}$  of $\cH$, identify $K$ with 
$$\psi = \sum_{i,j} K_{i,j}u_i\otimes u_j\ .$$
Then \eqref{AL5} follows from \eqref{AL2} and \eqref{AL3}. 
\end{proof}

\begin{proof}[Proof of Theorem~\ref{ALthm}] 
 Consider any purification $\rho_{123}$ of $\rho_{12}$; this exists by Lemma~\ref{purlm}. By  \eqref{AL0},  $S_{12} = S_3$ and 
 $S_{1} = S_{23} \leq S_2 + S_3$ and hence
$
S_{12} \geq S_1 - S_2$.
By symmetry, one then has (\ref{AL4}). \end{proof}

The cases of equality in the triangle inequality have been determined in \cite{CL12}. The following theorem is  the part of Theorem 1.3 in \cite{CL12} that we shall need below.

\begin{thm}\label{ALeqthm}  For arbitrary $a,b > 0$, there exist bipartite states $\rho_{12}$ such that
\begin{equation*}
S_{12} = S_1 -S_2
\end{equation*}
and
\begin{equation*}
S_{12} = a \qquad{\rm and}\qquad S_2 = b\ .
\end{equation*}
\end{thm} 

For the proof, see \cite{CL12}.   Note that the two conditional entropies associated to $\rho_{12}$ are
\begin{equation*}
S_{12} - S_2 = a-b \qquad{\rm and}\qquad S_{12} - S_1 = -b\ .
\end{equation*}
Thus, there exist bipartite states $\rho_{12}$ {\em that saturate the triangle inequality}  for which one of the conditional entropies is an arbitrary negative number, and the other is an arbitrary positive number.

\section{SSA and Entanglement}

Let $\rho_{12}$ be a density matrix 
on the tensor product of two Hilbert spaces
$\HH_1\otimes \HH_2$.  Then $\rho_{12}$ is {\em finitely
separable}
if and only if it has  a decomposition  as a convex combination of
tensor products:
\begin{equation}\label{sep}
\rho_{12} = \sum_{k=1}^n \lambda_k \rho_1^k \otimes \rho_2^k
\end{equation}
where the $\lambda_k$ are positive and sum to $1$, and each
$\rho_\alpha^k$ is a density matrix on $\HH_\alpha$.  A
bipartite state is {\em separable} if it is in the closure
of the set of finitely separable states.
A bipartite state that is not separable is 
{\em entangled}.  

For many purposes in quantum information theory, it is important to be able to decide whether a given bipartite state is entangled or not, already a challenging problem, and beyond this, to quantify the degree of entanglement of entangled states.  For more information in the importance entanglement in quantum information theory, see \cite{H4}. We now discuss the mathematics of one such measure of entanglement.

For any tripartite density matrix $\rho_{123}$ on $\HH_1\otimes\HH_2\otimes\HH_3$,   
the {\em conditional mutual information of $1$ and $2$ given $3$}, $I(1,2|3)$, is
defined by
\begin{equation}\label{CMI}
I(1,2|3) := S_{13}+S_{23} - S_{123} - S_3 \geq 0\ .
\end{equation}
Note that  SSA is equivalent to the fact that $I(1,2|3) \geq 0$

\begin{defi}[Squashed Entanglement]\label{sqent} The functional ${\rm E}_{\rm sq}$ on bipartite density matrices $\rho_{12}$ by
\begin{equation}\label{squashed}
{\rm E}_{\rm sq}(\rho_{12}) =  \frac12 \inf\{ \ I(1,2|3) \ : 
\rho_{123} \ {\rm is}\ any\ {\rm tripartite\ extension\ of }\
\rho_{12}\ \}\ .
\end{equation}
is known as the
{\em squashed entanglement} of $\rho_{12}$,  
\end{defi}

This functional was first introduced by Tucci \cite{Tu}, and  was rediscovered
by Christandl
and Winter \cite{CW1} who studied it and  proved that it has many important
properties, such as additivity. 

The cases of equality in SSA have been determined in \cite{Rus02,HJPW}, and it follows from the results in the second of these papers  that if $I(2,2|3) = 0$, then
$\HH_3$ has the form 
$$\HH_3 = \bigoplus_{j=1}^m \HH^j_{3\ell}\otimes \HH_{3r}^j$$ and 
$\rho_{123}$ has the form
\begin{equation}\label{eqssa}
\rho_{123} = \bigoplus_{j=1}^m \rho_{1,3\ell}^j\otimes \rho_{3r,2}^j\ . 
\end{equation}
Evidently, for any $\rho_{123}$ of the form (\ref{eqssa}),
$\rho_{12} := \tr_3(\rho_{123})$ is separable. Thus, if one knew that the infimum in Definition~\ref{sqent} was attained, 
it would follow that if ${\rm E}_{\rm sq}(\rho_{12}) =0$, then $\rho_{12}$ is separable.  On the other hand, let
$\rho_{12}$ be separable, and have the decomposition
(\ref{sep}).
Take $\rho^j_{1,3\ell}$ to be an
arbitrary
purification of $\rho_1^j$ onto $\HH_1\otimes
\HH_{3\ell}^j$, take $\HH_{3r}$ to be
one-dimensional with 
$\rho_{3r,2}^j = \rho_2^j$. Then
$\rho_{123}$ is an extension of the given separable
bipartite state $\rho_{12}$ for which equality holds in 
(\ref{SSA}).  Thus, whenever $\rho_{12}$ is finitely 
separable, then ${\rm E}_{\rm sq}(\rho_{12}) =0$, and then a continuity argument shows that whenever $\rho_{12}$ is separable, 
${\rm E}_{\rm sq}(\rho_{12}) =0$. However, since it is not known whether the infimum in \eqref{squashed} is attained, it is not so simple to see that if 
${\rm E}_{\rm sq}(\rho_{12}) =0$, then $\rho_{12}$ is separable. Nonetheless, this has been proved in \cite{BCY}, and thus it is now known that 
${\rm E}_{\rm sq}(\rho_{12}) =0$ if and only if $\rho_{12}$ is separable. 
That is, ${\rm E}_{\rm sq}$  provides  a {\em faithful} measure of entanglement. 

The following extension of SSA is proved in \cite{CL12}:

\begin{thm}[Extended SSA]\label{essa}
For all tripartite states $\rho_{123}$,
\begin{equation}\label{essa1}
 S_{13}+S_{23} - S_{123} - S_3 \geq 2 \max\{ S_1 - S_{12} \ ,  S_2 - S_{12}\ ,  0\ \} \ .
\end{equation}
\end{thm}

This has the immediate corollary that 
\begin{equation}\label{lbesf1}
{\rm E}_{\rm sq}(\rho_{12}) \geq   \max\{ S_1 - S_{12} \ ,  S_2 - S_{12}\ ,  0\ \} \ .
\end{equation}
Therefore if either of the conditional entropies $S_{12} - S_1$ or $S_{12} - S_2$ is strictly negative, then 
${\rm E}_{\rm sq}(\rho_{12})>0$, and $\rho_{12}$ is entangled.   

It is also shown in \cite{CL12} that the factor of $2$ on the right side cannot be replaced by any larger value. The argument uses a tripartite extension of a bipartite state that saturates the Araki-Lieb Triangle inequality.  This is natural since the proof of Theorem~\ref{essa}, which we give below, makes use of the purification arguments used to prove the Araki-Lieb Triangle inequality. The construction yields a tripartite state that, when inserted into the definition \eqref{squashed}, yields the upper bound
\begin{equation*}
{\rm E}_{\rm sq}(\rho_{12})  \leq  S_1 - S_{12} \qquad{\rm with}\qquad S_1 - S_{12} > 0\ .
\end{equation*}
Then from the lower bound, \eqref{lbesf1}, we have  ${\rm E}_{\rm sq}(\rho_{12})  =  S_1 - S_{12} > 0$. This proves that \eqref{lbesf1}, and hence the factor of $2$ in Theorem~\ref{essa}, is sharp.
We refer to \cite{CL12} for this construction.

A weaker form of the inequality (\ref{lbesf1}) was given by
Christandl and Winter
\cite{CW1}.
Their lower bound involves the
averaged  quantity 
\begin{equation}\label{wkr}
\frac12 (S_1+S_2) -S_{12}
\end{equation}
 in place of
$\max\{S_1-S_{12}, S_2-S_{12}\}$. The difference can be significant: By the remarks following Theorem~\ref{ALeqthm},
there exist states $\rho_{12}$  for which $\max\{ S_1 - S_{12} \ ,  S_2 - S_{12}\ ,  0\ \}$  is 
arbitrarily large, but the quantity in (\ref{wkr}) is negative. 
Moreover, the argument in \cite{CW1}
relied on a lower bound for the  the {\em one-way
distillable entanglement} 
in terms of the conditional entropy. This inequality, known
as the {\em hashing inequality}  had been a long-standing
conjecture,
and its proofs remain complicated.  Our contribution was to
show how this stronger lower bound follows in a relatively
simple manner from strong subadditivity, and to provide the
examples that prove the
sharpness of these bounds.

\begin{proof}[Proof of Theorem~\ref{essa}] Consider any purification $\rho_{1234}$ of $\rho_{123}$.
Then since $\rho_{1234}$ is pure,   $S_{23} = S_{14}$ and
$S_{124} = S_3$
Then
$$S_{12} + S_{23} - S_1 - S_3 = S_{12} + S_{14} - S_{124}-S_1\ ,$$
and the right hand side is non-negative by (\ref{SSA}). 
This proves 
\begin{equation*}
S_{12}+S_{23} \geq S_1+S_3\ ,
\end{equation*}
which may also be seen as a consequence of the convexity of the conditional entropy.

Next, adding  $S_{12}+S_{23} \geq S_1+S_3$ and   $S_{13}+S_{23} \geq S_1+S_2$, we obtain
\begin{equation*}
 S_{12} + S_{13} + 2S_{23}  \geq 2S_1 + S_2 + S_3 \ .
 \end{equation*}
Again consider any purification $\rho_{1234}$ of $\rho_{123}$.  Then we obtain, using $S_{12} = S_{34}$,
$S_{23} = S_{14}$, and $S_2 = S_{134}$,
$$
S_{13}+S_{34} -S_{134} -S_3 \geq 2(S_{1} - S_{14})\ ,
$$
which is (\ref{essa1}) with different indices.  
\end{proof}

\section{Minkowski's Inequality and SSA}

This section begins with a second  proof of SSA in the classical case taken from  \cite{CL99}.  Recall that a standard statement of the Minkowski inequality  is that for non-negative measurable
functions $f$ on the Cartesian product of two measure spaces $(X,\mu)$ and $(Y,\nu)$,
\begin{equation}\label{mink1}
\biggl(\int_X\biggl(\int_Y f(x,y){\rm d} \nu\biggr)^p{\rm d} \mu\biggr)^{1/p} \le
\int_Y\biggl(\int_X f^p(x,y){\rm d} \mu\biggr)^{1/p}{\rm d} \nu
\end{equation}
for $p\ge 1$, and that the opposite inequality holds for $0<p\leq 1$.  If $Y = \{1,\dots,n\}$ and $\nu$ is counting measure and we write $f_y(x)$ for $f(x,y)$, 
this reduces to 
\begin{equation*}
\left(\int_X\left(\sum_{j=1}^n f_j(x)\right)^{p}{\rm d}\mu\right)^{1/p}  \leq \sum_{j=1}^n \left(\int_X f_j^p{\rm d}\mu\right)^{1/p}\ .
\end{equation*}
Of course since $\left| \sum_{j=1}^n f_j(x)\right| \leq \sum_{j=1}^n |f_j(x)|$, and likewise for the integral version, the assumption of non-negativity can by dropped by inserting absolute values, but for our purposes, the present statements are more useful. Note in particular  that in \eqref{mink1}, there is equality at $p=1$ by Fubini's Theorem. 

The inequality \eqref{mink1} has a trivial extension to functions of three 
variables, and this extension has an interesting consequence, namely  classical SSA. 
Consider a non-negative measurable function $f(x,y,z)$ on the Cartesian
product of three measure spaces $(X,\mu)$, $(Y,\nu)$ and $(Z,\rho)$, and simply
holding $z$ fixed as a parameter, one gets
\begin{equation*}
\biggl(\int_X\biggl(\int_Y f(x,y,z){\rm d} \nu\biggr)^p{\rm d} \mu\biggr)^{1/p} \le
\int_Y\biggl(\int_X f^p(x,y,z){\rm d} \mu\biggr)^{1/p}{\rm d} \nu
\end{equation*}
{\it pointwise in $z$} for $p\ge 1$.
Integrating in $z$ then yields
\begin{equation}\label{mink2b}
\int_Z\biggl(\int_X\biggl(\int_Y f(x,y,z){\rm d} \nu\biggr)^p{\rm d} \mu\biggr)^{1/p}{\rm d} \rho \le
\int_Z\int_Y\biggl(\int_X f^p(x,y,z){\rm d} \mu\biggr)^{1/p}{\rm d} \nu{\rm d} \rho
\end{equation}
for $p\ge 1$, and of course the inequality reverses for $0<p\le1$.  

For any probability density $g$ on any finite measure space $(X,\mu)$ that belongs to $L^p(X,\mu)$ for some $p>1$,   $g \log g \in L^1(X,\mu)$ and 
\begin{equation}\label{mink2a}
\lim_{p\downarrow 1} \frac{1}{p-1}\left(\int_X g^p{\rm d}\mu -1\right) = \int_{X} g \log g {\rm d}\mu = -S(g)\ ,
\end{equation}
a formula that is the basis Irving Segal's approach to entropy \cite{Se60} in the classical and quantum setting.

To avoid technical details that never arise in the quantum setting, where the entropy is always non-negative, let us suppose that the measures $\mu$, $\nu$ and $\rho$ in \eqref{mink2b} are all finite.  Let $f_{123}$ be a probability density on $(X\times Y \times Z, \mu\otimes\nu\otimes \rho)$ that belongs to $L^p$ for some $p > 1$. Then by  \eqref{mink1}, all of its marginals are probability densities whose $p$th powers are integrable. 
Denote the various marginal densities of $f_{123}$ as follows:
$$f_{23}(y,z) = \int_X f(x,y,z){\rm d} \mu\ ,\quad f_{13}(x,z) = \int_Y f(x,y,z){\rm d} \nu\ , \quad
f_{3}(z) = \int_X\int_Y f(x,y,z){\rm d} \mu{\rm d} \nu\ .$$
Then inserting $f_{123}$ into \eqref{mink2b}, it becomes:
\begin{equation}\label{mink7a}
\int_Z\left(\int_X f_{13}^p {\rm d}\mu\right)^{1/p} \leq \int_Y\int_Z \left(\int_X f_{123}^p{\rm d}\mu\right)^{1/p}{\rm d}\nu{\rm d}\rho\ .
\end{equation}
Since \eqref{mink7a} is an equality at $p=1$, we get another inequality by taking the right derivative of both sides of \eqref{mink2b} with respect to $p$ at $p=1$, and by  \eqref{mink2a} this yields an entropy inequality which is, in the now familiar notation, with all terms finite,
$$S_{13}+ S_{23} \geq S_{123} + S_3\ .$$
That is, differentiating Minkowski's inequality for three measure spaces yields  the strong subadditivity of the classical entropy.  The condition that the measures are finite can be easily relaxed, but this is irrelevant for our present purposes. This proof of classical SSA has the merit of making no reference to conditional probability densities, and so one might hope that it would extend to the quantum setting. This is indeed the case.

The operator analog of \eqref{mink1} is elementary, though it appears the first proof was given in \cite{CL99}: 

\begin{thm}\label{minkthm2}Let $A$  be a positive operator on the tensor
product of two 
Hilbert spaces $\cH_1\otimes\cH_2$. Then for all $p\ge 1$,
\begin{equation}\label{mink5}
\bigl(\tr(\tr_1 A)^p\bigr)^{1/p} \le \tr\bigl(\bigl(\tr_ 2 A^p\bigr)^{1/p}\bigr)
\end{equation}
and inequality \eqref{mink5} reverses for $0<p\le 1$.
\end{thm} 

\begin{proof} Take $p>1$, and let $1/q = 1-1/p$. Then there is a positive operator $B\in \cB(\cH_2)$ such that $\tr_2 B^q =1$ and 
$$
\bigl(\tr(\tr_1 A)^p\bigr)^{1/p}  = \tr B(\tr_1A) = \tr [(I\otimes B)A]  = \sum_{i,j} \langle u_i\otimes v_j, (I\otimes B)A  \langle u_i\otimes v_j\rangle
$$
for any orthonormal bases $\{u_1,\otimes u_m\}$ and $\{v_1,\dots,v_n\}$ of $\cH_1$ and $\cH_2$ respectively.   (See the remarks on Schatten trace norms following Remark~\ref{conan} below for the relevant facts about duality.)  Choosing 
$\{v_1,\dots,v_n\}$ to be an eigenbasis of $B$,  $Bv_j = \lambda_jv_j$, we then have
$$
\bigl(\tr(\tr_1 A)^p\bigr)^{1/p} =   \sum_{i,j}\lambda_j \langle u_i\otimes v_j, A  ( u_i\otimes v_j)\rangle \leq \sum_{i}\left(\sum_{j} 
\langle u_i\otimes v_j, A  ( u_i\otimes v_j)\rangle^p\right)^{1/p}
$$
by H\"older's inequality and  $\tr B^q =1$.   Then by the Spectral Theorem and Jensen's inequality, in its standard form  for convex functions on $\R$,
$$\langle u_i\otimes v_j, A  \langle u_i\otimes v_j\rangle^p  \leq  \langle u_i\otimes v_j, A^p  \langle u_i\otimes v_j\rangle\ ,$$
and thus  $\bigl(\tr(\tr_1 A)^p\bigr)^{1/p}  \leq \sum_{i} \langle u_i, \tr_2A^p u_i\rangle^{1/p}$.  Now choose $\{u_1,\dots,  u_m\}$ to consist of eigenvectors of $\tr_2 A^p$ to conclude that \eqref{mink5} is valid for all $p\geq 1$. \

For $0<p\le 1$, define $r=1/p$ and $B = A^p$ so that 
$A = B^r$. Since $r>1$, the inequality proved above says
$\tr\bigl(\bigl(\tr_2B^r\bigr)^{1/r}\bigr) \geq \bigl(\tr\bigl(\tr_1B\bigr)^r\bigr)^{1/r}$.
This is \eqref{mink5} for $0<p< 1$ in other notation. 
\end{proof} 

The following theorem, also from \cite{CL99,CL08} does not have such a simple proof:

\begin{thm}\label{minkthm3} Let $A$  be a positive operator on the tensor
product of three 
Hilbert spaces $\cH_1\otimes\cH_2\otimes\cH_3$. Then
\begin{equation}\label{mink25}
\tr_3\bigl(\tr_2(\tr_1 A)^p\bigr)^{1/p} \le \tr_{13}\bigl(\bigl(\tr_2 A^p\bigr)^{1/p}\bigr)
\end{equation}
for $1 \leq p \leq 2$  while the reverse inequality holds for $0<p\le 1$, and neither inequality is valid for $p>2$.
\end{thm}

Evidently there is equality at $p=1$, and if $A$ is a tripartite density matrix $\rho_{123}$ then for all $p>0$,
$$
\frac{1}{p-1}\left( \tr_3\bigl(\tr_2(\tr_1 \rho_{123})^p\bigr)^{1/p} -1\right)  \leq \frac{1}{p-1} \left( \tr_{13}\bigl(\bigl(\tr_2 \rho_{123}^p\bigr)^{1/p}\bigr) -1 \right)
$$
Taking the limit $p\to 1$, we obtain $S_{13}+ S_{23} \geq S_{123} + S_3$, giving another proof of SSA in the quantum case. 

Note that while \eqref{mink5} is valid for all $p\geq 1$,  \eqref{mink25} is valid only for $1 \leq p \leq 2$.   This is an indication of the fact that the 
deeper notion of operator convexity comes into play in the proof of Theorem~\ref{minkthm3}, while the proof of 
Theorem~\ref{minkthm2} only required the use of Jensen's inequality for real valued functions on the real line.   Two references, \cite{CL99,CL08}, are given  for 
Theorem~\ref{minkthm3} because in \cite{CL99} the theorem was proved only for $0 < p \leq 1$ and $p=2$.  This is still enough to prove SSA -- simply take the left derivative.  

The fundamental input to the proof in \cite{CL99} came from a theorem of Epstein \cite{E73}, who proved the one conjecture  explicitly stated by Lieb in \cite[p. 282]{L73}. Lieb's conjecture was that:  For each fixed $B\in M_n^+(\C)$, and each $m\in \N$, the function
\begin{equation}\label{mink31}
A \mapsto \tr[(BA^{1/m}B)^{m}]
\end{equation}
is concave on $M_n^+(\C)$. 

Of course the case $m=1$ is trivial, but the case $m=2$ is equivalent to Wigner-Yanase Theorem, by what has explained at the close of Section 2. 
Thus, there had been two conjectured extensions of the Wigner-Yanase Theorem, namely the one that appears to have been put forward by 
Dyson, and was confirmed with the proof of the Lieb Concavity Theorem, and then Lieb's conjecture, and it was not immediately clear 
that they were equivalent. After all, Lieb proved the Lieb Concavity Theorem, and made the other conjecture in the very same paper. As explained later in this section, the conjecture is not only true, but it is equivalent to the Lieb Concavity Thoerem.

The proof of Lieb's conjecure, and somewhat  more, is due to Epstein  \cite{E73}:
\begin{thm}[Epstein's Theorem]\label{Ep73}
For  any $n\times \ell$ matrix $B$, and all $0 < p <1$, the function
\begin{equation}\label{mink32}
A \mapsto \tr[(B^*A^{p}B)^{1/p}]
\end{equation}
is concave on $M_n^+(\C)$. 
\end{thm} 

Epstein's proof relied on the theory of Herglotz functions; i.e., functions that are analytic on the upper half plane in $\C$ with a positive imaginary part. This includes $z \mapsto z^p$, but only  for $0 \leq p \leq1$. Using Epstein's result, Theorem~\ref{minkthm3} was proven in \cite{CL99} for $0 \leq p \leq 1$ and the special case $p=2$.   The part of Theorem~\ref{minkthm3} referring to  $1 < p < 2$ was left as a conjecture in \cite{CL99}. 

Then in \cite{CL08}, a duality method was introduced that led to the proof of Theorem~\ref{minkthm3} in full, and much more, including a new and simple proof of Epstein's Theorem. One of the new tools that \cite{CL08} brought to bear on matrix convexity and concavity theorems is the following:

\begin{lm}\label{rock} Let $X$ and $Y$ be  vector spaces. 
If $f(x,y)$ is jointly convex in $(x,y) \in X\times Y$ with values in $(-\infty, \infty]$,  then $g(x) := \inf_{y\in Y}\{  f(x,y) \}$ is convex.  If $f(x,y)$ is jointly concave in 
$(x,y) \in X\times Y$ with values in $[-\infty,\infty)$,  then $g(x) := \sup_{y\in Y}\{ f(x,y)\}$ is  concave.  
 \end{lm}

A proof of this lemma, which has many uses other than the ones made here, may be found in \cite[Theorem 1]{R74}. As in \cite{CL08}, we give the simple proof for completeness. 

\begin{proof}[Proof of Lemma~\ref{rock}]
 For any $x_0$ and $x_1$ with $g(x_0),g(x_1) < \infty$, and any $0<\lambda<1$,
 pick $\varepsilon>0$, and choose $y_0$ and $y_1$ so that
$$f(x_0,y_0) < g(x_0)+\varepsilon \qquad{\rm and}\qquad  f(x_1,y_1) <
g(x_1)+\varepsilon\ .$$ Then \begin{eqnarray} \qquad\qquad \qquad g(
(1-\lambda)x_0+\lambda x_1) &\le& f((1-\lambda)x_0+\lambda x_1,
(1-\lambda)y_0+\lambda y_1)\nonumber\\ &\le& (1-\lambda)f(x_0,y_0)+\lambda
f(x_1,y_1)   \nonumber \\ &\le&  (1-\lambda)g(x_0) + \lambda g(x_1) +
\varepsilon\ . \qquad\qquad\qquad \qquad\qquad\qquad \nonumber
\end{eqnarray}
\end{proof}

There is another well-known  general convexity/concavity lemma \cite{R70} that is useful in conjunction with Lemma~\ref{rock}:

\begin{lm}\label{conscal}   Let $V$ be a vector space and $f: V \to[0,\infty)$  be homogeneous of degree $q>0$, so that for all $s>0$ and all $x\in V$, $f(sx) = s^qf(x)$. Suppose also that $f(x) >0$ for $x\neq 0$.  Then if $f$ is concave, so is $f^{1/q}$. 
\end{lm}

\begin{proof}  First consider the case $q=1$. We claim that $f$ is concave if and only if $\mathcal{K} := \{x: f(x) \geq 1\}$ is convex. Suppose $f$ is concave. 
Let $x,y\in \mathcal{K}$, and $0 < \lambda < 1$. Then
$$f(\lambda x + (1-\lambda y) \geq \lambda f(x) + (1-\lambda)f(y) \geq 1\ ,$$
and hence $\mathcal{K}$ is convex. Now suppose that $\mathcal{K}$ is convex. Let $x,y\in V$. We must show $f(x+y) \leq f(x) + f(y)$. This is obvious if either $x=0$ or $y=0$, so we suppose both are non-zero.  Then
$$
\frac{f(x+y)}{f(x)+f(y)} = f\left( \frac{f(x)}{f(x)+f(y)} \frac{x}{f(x)} + \frac{f(y)}{f(x)+f(y)} \frac{y}{f(y)}\right) \ .
$$
Evidently $x/f(x)\in \mathcal{K}$ and $y/f(y)\in \mathcal{K}$ so by the convexity of $\mathcal{K}$, $f(x+y) \geq f(x) + f(y)$, and then since $f$ is homogeneous of degree one, $f$ is concave.  This proves the initial claim.

Now suppose that $f$ is homogeneous of degree $q$, and define $\mathcal{K}$ as before. The first part of the argument given above did not involve the homogeneity, and hence it still yields the result that $\mathcal{K}  = \{x: f(x) \geq 1\}$ is convex. But evidently
$$
\{x: f(x) \geq 1\}  =  \{x: f(x)^{1/q} \geq 1\}  
$$
and then since $f^{1/q}$ is homogeneous of degree one, it is concave. 
\end{proof}

\begin{remark}\label{conan} The  obvious analog  of Lemma~\ref{conscal} for convex functions may be proved in the same manner. 
\end{remark}

Finally we recall some well known facts about the {\em Schatten Trace Norms} on $M_n(\C)$. For $p \geq 1$ and $A\in M_n(\C)$, define
$$\|A\|_p := \left(\tr[ (A^*A)^{p/2}]\right)^{1/p}\ .$$
Define $\|A\|_\infty$ to be the operator norm of $A$. 
Evidently, if $\sigma_1(A),\dots, \sigma_n(A)$ are the singular values of $A$, $\|A\|_p = (\sum_{j=1}^n \sigma_j(A)^p)^{1/p}$. Elementary proofs of the following statements may be found in \cite{C10,S05}:  For all $1 \leq p \leq \infty$, $A \mapsto \|A\|_p$ is a norm, and for all $p$ and $q$ with  $1/p + 1/q = 1$,
\begin{equation}\label{basedual}
\|A\|_p = \sup\{ |\tr[AB]|\ :\ B\in M_n(\C)\ ,\ \|B\|_{q} = 1\}  \ .
\end{equation}
In fact, it is easy to see that for $1 < p < \infty$,  if $A \in M_n^{+}(\C)$, and $B := \|A\|_p^{1-p} A^{p-1}$, $\|B\|_q = 1$, and $\tr[AB] = \|A\|_p$. 
Thus,  
 \begin{equation}\label{basedual2}
{\rm For}\ A\in M_n^{+}(\C),\ \qquad \|A\|_p = \sup\{ \tr[AB]\ :\ B\in M_n^{+}(\C)\ ,\ \|B\|_{q} = 1\}\ .
\end{equation}

We now show that using Lemma~\ref{rock} and  Lemma~\ref{conscal}, Epstein's Theorem follows directly from the Lieb Concavity Theorem.

\begin{proof}[Proof of Theorem~\ref{Ep73}] Let $0 < p < 1$, $A\in M_n^+(\C)$, and let $B$ be any complex $n\times m$ matrix.
\begin{equation}\label{fdef} 
f(A) := \left(\tr[ (B^*A^p B)^{1/p}]\right)^p = \|B^*A^p B\|_{1/p}\ .
\end{equation}
Then by \eqref{basedual2}, 
\begin{eqnarray*}
f(A) &=& \sup\{\ \tr[(B^*A^pB)Z]  \ : \ Z\in M_n^+(\C)\ , \tr[Z^{1/(1-p)}] =1\ \}\\
&=&   \sup\{\ \tr[(B^*A^pB)Y^{1-p}]  \ : \ Y\in M_n^+(\C)\ , \tr[Y] =1\ \}
\end{eqnarray*}
By the Lieb Concavity  Theorem, $(A,Y) \mapsto \tr[(B^*A^pB)Y^{1-p}]$  is jointly concave. By Lemma~\ref{rock}, $f$ is concave. Since evidently $f$ is homogeneous of degree $p$, by Lemma~\ref{conscal}, $f^{1/p}(A)= \tr[ (B^*A^p B)^{1/p}]$ is concave.
\end{proof} 

The same methods may be applied with the Ando Convexity Theorem to prove that for $1 \leq p \leq 2$, and fixed $n\times m$ matrix $B$, 
$A \mapsto \tr[ (B^*A^p B)^{1/p}]$\ is convex, as was shown in \cite{CL08} by a slightly different argument. A proof along the above lines could be given here, but it would require the reverse H\"older inequality for traces, and would lead to somewhat less than what was proved in \cite{CL08}. Therefore, we simply quote the main result obtained by this method  in \cite{CL08}.

\begin{thm}\label{main} 
For all $1  \le p \le 2$,  and for all $ q\geq 1$, and $n\times m$ complex matrices $B$ define
$$
\Upsilon_{p,q}(X) = {\rm Tr}\left[(B^*X^pB)^{q/p}\right]\ .
$$
Then $\Upsilon_{p,q}$  is convex on $M_n^+(\C)$, while 
for $0 \le p \le q \le 1$,  $\Upsilon_{p,q}$  is concave on $M_n^+(\C)$.
For $p>2$, $\Upsilon_{p,q}$ is neither convex nor concave for any values of $q\neq p$. 
\end{thm}

\begin{remark} The concavity of $\Upsilon_{p,1}$  is Epstein's Theorem; the other cases were new in \cite{CL08}.
\end{remark}

 The following simple generalizaton of Epstein's Theorem  is useful for the proof of Theorem~\ref{minkthm3}:

\begin{cl}\label{cpcl}  Let $\Phi$ be any completely positive map from $M_n(\C)$ to $M_m(\C)$. Then for all $1  \le p \le 2$,  and for all $ q\geq 1$, 
\begin{equation}\label{cpcl1}
X \mapsto   \tr[(\Phi(X^p))^{q/p}]
\end{equation}
is convex on $M_n^{+}(\C)$.   For for $0 \le p \le q \le 1$, the mapping in \eqref{cpcl1} is concave. 
\end{cl}

\begin{proof} By the Stinespring Representation Theorem in its standard form \cite{Pa03,S55}, there is a representation $\pi$ of $M_n(\C)$ on $\C^m$ for some $m$, and an $m\times n$ matrix $B$ such that $\Phi(X)  = B^*\pi(X)B$. Then $\Phi(X^p) = B^*\pi(X^p)B = B^*(\pi(X))^pB$. Since $X \mapsto \pi(X)$ is linear, the claim now follows from Theorem~\ref{main}
\end{proof}

As an application of Corollary~\ref{cpcl} consider a tensor product of two finite dimensional Hilbert spaces $\cH_1\otimes \cH_2$, and take $\Phi$ to be the partial trace over $\cH_2$
; i.e., $\Phi(X) = \tr_2(X)$ so that  $\Phi: \cB(\cH_1\otimes \cH_2) \to \cB(\cH_1)$.  Then 
for all $1  \le p \le 2$,  and for all $ q\geq 1$, 
\begin{equation}\label{cpcl2}
X \mapsto  \Psi_{p,q} :=   \tr_1[(\tr_2(X^p))^{q/p}]  
\end{equation}
is convex on $M_n^{+}(\C)$.   For for $0 \le p \le q \le 1$, the mapping in \eqref{cpcl2} is concave.  

Further specializing to the case in which
$X\in \cB(\cH_1\otimes \cH_2)$ is a block diagonal  matrix with entries $X_1,\dots,X_k$  in $\cB(\cH_1)$; i.e.,
$$
X = \left[\begin{array}{cccc} X_1  & &\\   & \ddots & \\ & & X_k\end{array}\right]
$$
we see that  for all $1  \le p \le 2$,  and for all $ q\geq 1$, 
\begin{equation}\label{cpcl3}
(X_1,\dots,X_k) \mapsto \tr\left[ \left(\sum_{j=1}^k X_j^p\right)^{1/p}\right]
\end{equation}
is jointly convex while,  for $0 \le p \le q \le 1$, the mapping in \eqref{cpcl3} is concave.

\begin{proof}[Proof of Theorem~\ref{minkthm3}]   
Let the dimension of $\cH_1$ be $n$. Then with $\Psi_{p,1}$ denoting the functional defined in \eqref{cpcl2} with $q=1$, 
\begin{eqnarray}\label{cpcl4}
\tr_3\left[\tr_2(\tr_1A)^{p}\right]^{1/p} &=& \tr_{13}\left[\tr_2\left( \frac1n I_{1}\otimes \tr_1A\right)^{p}\right]^{1/p} \nonumber\\ 
&=& \Psi_{p,1}\left(\frac1n I_{1}\otimes {\rm Tr}_1 A\right)
\end{eqnarray}
where the pair of spaces in the definition of $\Psi_{p,1}$ is taken to be
$\cH_1\otimes\cH_3$ and $\cH_2$, and where $I_1$ is the identity on $\cH_1$. By Uhlmann's Lemma, Lemma~\ref{ullm}, there is a finite group $\mathcal{G}$ of unitaries $W$ on $\cH_1$ so that
\begin{equation}\label{cpcl5}
\frac1n I_{1}\otimes \tr_1 A  = \frac{1}{|\mathcal{G}|}\sum_{W\in \mathcal{G}} (W^*\otimes I_{23})A  (W\otimes I_{23})\ .
\end{equation}
Since each $W$ is unitary, with $I_{2,3}$ denoting the identity on $\cH_2\otimes\cH_3$,
\begin{equation}\label{cpcl6}
(W^*\otimes I_{23}) A  (W\otimes I_{23})^p = (W^*\otimes I_{23}) A^p  (W\otimes I_{23})
\end{equation}
and then 
\begin{eqnarray}\label{cpcl7}
\Psi_{p,1}\left((W^*\otimes I_{23})A\, (W\otimes I_{23})\right) &=&
\tr_{13}\left[ \left(\tr_2\left[(W^*\otimes I_{23})A^p \, (W\otimes I_{2,3})\right]\right)^{1/p} \right]\nonumber \\
&=&
\tr_{13}\left[  ( W^*\otimes I_3) \left( \tr_2(A^p) \right)^{1/p} (W\otimes I_3)\right] \nonumber \\
&=&  \tr_{13}\left[  \left( \tr_2[A^p] \right)^{1/p} \right]\ .
\end{eqnarray}
By \eqref{cpcl4} , \eqref{cpcl5} and the convexity of $\Psi_{p,1}$ for $1 \leq p \leq 2$, \eqref{cpcl6} and finally  \eqref{cpcl7},
\begin{eqnarray*}
\tr_3\left[\tr_2(\tr_1A)^{p}\right]^{1/p} &=& \Psi_{p,1}\left(  \frac{1}{|\mathcal{G}|}\sum_{W\in \mathcal{G}} (W^*\otimes I_{23}) A  (W\otimes I_{23})   \right)\\
&\leq& \frac{1}{|\mathcal{G}|}\sum_{W\in \mathcal{G}} \Psi_{p,1}\left(   (W^*\otimes I_{23}) A  (W\otimes I_{23})  \right)= 
 \tr_{13}\left[  \left( \tr_2[A^p] \right)^{1/p} \right]\ .
\end{eqnarray*}
This proves \eqref{mink25} for $1 \leq p \leq 2$, and the reverse inequality for $0 < p \leq 1$ is proved in the same manner using the concavity of $\Psi_{p,1}$ for such $p$. 
\end{proof} 

The duality method described here was further developed by Zhang \cite{Z} to prove a conjecture raised in \cite{CFL18} that generalized  a conjecture of Audenaert and Datta \cite{AD15}  that in turn generalized some conjectures put put forward  in \cite{MLDSFT,WWY} concerning  {\em sandwiched R\'enyi entropies}.  For density matrices $\rho,\sigma\in M_n^{++}(\C)$, and $\alpha>0$, $\alpha \neq 1$, the  sandwiched R\'enyi relative entropy
$D_\alpha(\rho||\sigma)$ is defined by
\begin{equation}\label{SRE1}
D_\alpha (\rho||\sigma)  = \frac{1}{\alpha -1}\log\left( \tr[(\sigma^{(1-\alpha)/2\alpha} \rho \sigma^{(1-\alpha)/2\alpha})^\alpha]\right)\ . 
\end{equation}
It was conjectured in \cite{MLDSFT,WWY} that for any completely positive unital map $\Phi$, 
\begin{equation}\label{SRE3}
D_\alpha (\Phi^\dagger \rho||\Phi^\dagger \sigma)  \leq D_\alpha (\rho||\sigma)\ , 
\end{equation}
which is the DPI for sandwiched R\'enyi relative entropies,  and proven there for $1 <\alpha \leq 2$ as a consequence of the Lieb Concavity Theorem.   Work on the remaining cases led to  conjectures in \cite{AD15} and \cite{CFL18} concerning the joint convexity and concavity properties of the function
\begin{equation}\label{SRE2}
\Psi_{p,q,s}(A,B) = \tr[ (B^{q/2} K^* A^p K B^{q/2})^s]
\end{equation}
on $M_n^{++}(\C)\times  M_n^{++}(\C)$. Note that if one takes $K=I$, $A = \rho$, $B = \sigma$, $s= \alpha$, $p= 1$ and $q = (1-\alpha)/2\alpha$, the trace in \eqref{SRE2} becomes the trace in \eqref{SRE1}.  

The conjecture raised in \cite{AD15} and slightly generalized  in \cite{CFL18} concerns the set of parameter values $p$, $q$, $s$  
for which the function $\Psi_{p,q,s}(A,B)$ is jointly convex or concave. By the usual arguments, such as described here in Section 3, one deduces monotonicity from these concavity/convexity properties wherever one has them. Very soon after the work in \cite{MLDSFT,WWY}, Beigi \cite{B13} and Frank and Lieb \cite{FL13} simultaneously, but by quite different arguments, proved the monotonicity of the sandwiched R\'enyi relative entropy for all $\alpha > 1$. (The papers \cite{B13}, \cite{FL13} and  \cite{MLDSFT} all appear in volume 54 of Journal of Mathematical Physics.)
Investigation of $\Psi_{p,q,s}(A,B)$ continued, and many partial results had been obtained by Hiai, 
who further developed Epstein's method, and by myself, Frank and Lieb using mainly the duality method described here. See \cite{CFL18} for 
references to this earlier work leading up to Zhang's full solution. 

The key to Zhang's work is his beautiful Theorem 3.3 which gives two families of variational expressions 
for $\tr[|XY|^r]$, $X,Y$ invertible in $M_n(\C)$. With this device, which plays the role of \eqref{basedual2} in the proof given here of 
Theorem~\ref{Ep73}, but greatly extends its power,  he very efficiently settled all outstanding cases of the conjecture, as well as giving a new, 
unified  proof of all the cases that had been proved by various means. 

The results obtained so far in this section permit simple proofs to be given of two more of the theorems proved in \cite{L73}.  We have seen how the Lieb Concavity Theorem and duality can be 
used to prove Lieb's conjecture  \cite[p. 282]{L73}, so that for all  $m\in \N$, $B\in M_n^{++}$, we know that $X \mapsto \tr[(BX^{1/m}B)^m]$ is 
concave on $M_n^{++}(\C)$.  Replace $B$ by $B^{1/2m}$ and write $B = e^H$. Then for each $m$,
$$
X \mapsto \tr \left[(e^{H/m} e^{\log X/m})^m\right]
$$
is concave. By the Trotter Product Formula \cite{T59}, so is
$$
X \mapsto \tr\left [e^{H + \log X}\right] = \lim_{m\to \infty}\tr \left[(e^{H/m} e^{\log X/m})^m\right]\ .
$$
This proves the fourth theorem in \cite{L73}:

\begin{thm}[Lieb, 1973]\label{explog}
For all self-adjoint $H\in M_n(\C)$, the function
\begin{equation}\label{explog1}
X \mapsto \tr\left [e^{H + \log X}\right] 
\end{equation}
is concave.
\end{thm}
Another interesting proof that uses the duality method of \cite{CL08} to deduce Theorem~\ref{explog} from the joint convexity of the relative entropy was given by Tropp \cite{Tr12}.  For a monotonicity variant of Theorem~\ref{explog} that is strictly stronger, see \cite{C22}.

From here, it is simple to deduce the fifth theorem in \cite{L73}, the Triple Matrix Theorem that was used in the second proof given here of SSA.  
\begin{thm}[Lieb, 1973]\label{TripleMatrix}
For all self-adjoint $H,K,L\in M_n(\C)$, 
\begin{equation}\label{TM1}
\tr[e^{H+K+L}]   \leq \tr[e^{H}T_{e^{-K}}(e^L)]] \ ,
\end{equation}
where
\begin{equation}\label{TM2}T_A(B) = \int_0^\infty \frac{1}{s+A} B \frac{1}{s+A} {\rm d}s  =  \frac{{\rm d}}{{\rm d}t} \log(A + t B)\bigg|_{t=0}\ .
\end{equation}
\end{thm}

\begin{proof}[Proof of Theorem~\ref{TripleMatrix}]
The second equality in \eqref{TM2} follows from the integral representation of the logarithm, \eqref{logrep} given below. The function \eqref{explog1}
is  homogeneous of degree one as well as concave, so that Lemma~\ref{lieblem} may be applied, with the result that 
$$
\frac{{\rm d}}{{\rm d}x} \tr[\exp( H+K + \log(e^{-K} + x e^L))]\bigg|_{x=0} \geq \tr[e^{H+K+L}]. 
$$
Since
$$
\frac{{\rm d}}{{\rm d}x} \log (e^{-K} + xe^L)\bigg|_{x = 0}  =  T_{e^{-K}}(e^L)\ ,
$$
with $T_A(B)$ defined as in \eqref{TM2}, this proves \eqref{TM1}. 
\end{proof}

As noted in \cite{L73},  the second part of Lemma~\ref{lieblem} allows one to reverse the argument that led from Theorem~\ref{explog} to Theorem~\ref{TripleMatrix}, and hence these two  theorems are equivalent to one another.   In fact, now that the five main theorems of \cite{L73} have been stated, it is a convenient point to discuss their equivalence, which is proved in the final section of \cite{L73}. In fact, one can expand the list of equivalent results to include Epstein's Theorem and the various equivalent statements of SSA. 

We have seen in this section that Theorem~\ref{L2} $\Rightarrow$ Theorem~\ref{Ep73}  $\Rightarrow$ Theorem~\ref{explog} $\iff$ Theorem~\ref{TripleMatrix}.
Now following \cite{L73}, we show that Theorem~\ref{explog}  $\Rightarrow$ Theorem~\ref{L3} $\Rightarrow$ Theorem~\ref{L2}. 

The starting point is once again Lemma~\ref{lieblem}, this time applied to the function figuring in Theorem~\ref{L3} in the case $X =Y$ and with $K$ self-adjoint, which we write as
\begin{equation}\label{TM3}
Q(X,K) := \int_0^\infty \tr\left[K\frac{1}{s+X} K \frac{1}{s+X}\right]{\rm d}s = \tr[K T_X(K)]\ ,
\end{equation}
where $T_X(K)$ is given by \eqref{TM2}.
The function $Q(X,K)$ is homogenous of degree one, and one readily computes that for all  positive definite $Y$ and self-adjoint  $H$, 
$$
\frac{{\rm d}}{{\rm d}t} Q(X+tY,K+tH)  = -\tr[Y R_X(K)] + 2\tr[H T_X(K)]
$$
where
\begin{equation}\label{TM4}
R_X(K) = 2\int_0^\infty \frac{1}{s+X} K \frac{1}{s+X}K  \frac{1}{s+X} {\rm d}s =  -\frac{{\rm d}^2}{{\rm d}t^2} \log(X+ tK)\bigg|_{t=0}\ .
\end{equation}

By Lemma~\ref{lieblem}, the function $Q(X,K)$  is jointly convex if and only if  for all self-adjoint $H$ and all positive definite $Y$, 
\begin{eqnarray*}
0 &\geq& -\tr[Y R_X(K)] + 2\tr[H T_X(K)]  - \tr[HT_Y(H)]\\
&=&  - \tr[Y R_X(K)]   +  \tr[T_X(K) T_Y^{-1}(T_X(K)]\\
 &-&  \tr[(H- T_Y^{-1}(T_X(K))T_Y(H- T_Y^{-1}(T_X(K))] \ ,
\end{eqnarray*}
where we have simply completed the square. 
Taking $H = T_Y^{-1}(T_X(K)$, it is evident that  $Q(X,K)$ is jointly convex if and only if  for all  $X,Y\in M_n^{++}(\C)$ and $K=K^*\in M_n(\C)$,
\begin{equation}\label{TM5}
\tr[Y R_X(K)]  -\tr[T_X(K) T_Y^{-1}(T_X(K)] \geq 0 \ .
\end{equation}
By Theorem~\ref{L3}, \eqref{TM5} is valid.  With this in hand, the concavity asserted in Theorem~\ref{explog} is easily proved: 

Since the exponential and logarithm functions are inverse to one another, it follows from \eqref{TM2} that 
${\displaystyle T_Y^{-1}(H) = \frac{{\rm d}}{{\rm d}t} \exp(Y + t H)\bigg|_{t=0}}$, and hence setting $B := e^{L+ \log X}$
a simple computation yields
\begin{equation}\label{TM5B}
\frac{{\rm d^2}}{{\rm d}t^2}  \tr [ \exp(L  + \log(X+tK))] = -\tr[ B R_X(K)] + \tr[T_X(K) T_B^{-1}(T_X(K))] \leq 0\ .
\end{equation}
 Therefore $X \mapsto \tr[e^{L + \log X}]$ is concave, and the argument just presented is Lieb's original proof of  Theorem~\ref{explog}. The argument just presented shows that  that  since \eqref{TM5} and \eqref{TM5B} are equivalent, Theorem~\ref{explog} is equivalent to the joint  convexity of $Q(X,K)$.
Now the discussion around \eqref{passageback} made in connection the the Wigner-Yanase Theorem allows one to recover the full strength of Theorem~\ref{L3} from  the joint convexity of $Q(X,K)$. Thus, 
Theorem~\ref{L3} $\iff$ Theorem~\ref{explog}. 

Finally, as in \cite{L73}, if one writes $Q(X,Y,K)$ for the functional that is asserted to be jointly convex in Theorem~\ref{L3},  one has by simple computations that for all $0 < p,q$, $p+q < 1$, 
$$
\int_{0}^{\infty }t^{p-1}\int_{0}^{\infty} u^{q-1} Q(I + tA, I+uB,K) {\rm d}t{\rm d}u\ ,
$$
is a constant multiple of the function appearing in Theorem~\ref{L2}, and the case $p+q =1$ follows as a limiting case. This proves that Theorem~~\ref{L3} 
$\Rightarrow$ Theorem~\ref{L2}. 

Furthermore, the simple argument of Tropp \cite{Tr12} using the duality method  discussed here shows that the joint convexity of the relative entropy, Theorem~\ref{Lind}, implies
Theorem~\ref{explog}. Therefore, Theorem~\ref{L1} $\Rightarrow$ Theorem~\ref{Lind} $\Rightarrow$ Theorem~\ref{explog} $\Rightarrow$ SSA
$\Rightarrow$   Theorem~\ref{Lind},  where the final implication comes from Theorem~\ref{SSAmain}. 

It remains to be shown here  that Theorem~\ref{L1} $\iff$ Theorem~\ref{L2}, and that these theorems are true.  We shall deduce this all at once, with much more, in Section~12 from a beautiful theorem of Hiai and Petz \cite{HP12} that has a very simple proof.

\section{Monotone metrics}

In the Introduction, it was noted that  Theorems 2 and 3 of \cite{L73} provided the answers to questions that would not be asked for years to come, and that when they were asked, it was not recognized
that the answers could be found in \cite{L73}.  The questions concerned {\em monotone metrics} on the space of non-degenerate density matrices.  

 Let $\rho(t)$ be a differentiable path in the space $\mathfrak{S}$ of invertible $n\times n$ density matrices; i.e., elements of 
 $M_n^{++}(\C)$ with unit trace. Then the derivative $\rho'(t) = K(t)$ is a self-adjoint operator with $\tr[K(t)] =0$.  We may think of 
 $\mathfrak{S}$ as a differentiable manifold, and it is then of interest to equip it with various Riemannian metrics  that have the property that the 
 distance between $\rho_1,\rho_2\in \mathfrak{S}$, $d(\rho_1,\rho_2)$,   decreases under the application of
any quantum operation; i.e., 
\begin{equation}\label{che1}
d(\Phi^\dagger(\rho_1),\Phi^\dagger(\rho_1)) \leq d(\rho_1,\rho_2)\ .
\end{equation}
In such a metric, any quantum operation performed on the states can only make it harder to distinguish between them. See the Introduction to \cite{CFL18} for further discussion of the problem of distinguishing between states, which is basic to quantum communication. 

We may identify the tangent space to $\mathfrak{S}$ at each point $\rho$ with the space of traceless self-adjoint $n\times n$ matrices. If we denote the  quadratic form that specifies a Riemannian metric at $\rho$ by  $\gamma_\rho(K,K)$, then the contractive property \eqref{che1} for such a metric  is equivalent to
\begin{equation*}
\gamma_{\Phi^\dagger(\rho)}(\Phi^\dagger(K),\Phi^\dagger(K)) \leq \gamma_\rho(K,K)\ .
\end{equation*}

The same question had been of interest  and answered in the classical case where  in which the analog of  $\mathfrak{S}$ is the set  $\mathfrak{S}_c$ of $n$ dimensional strictly positive probability vectors 
$p := (p_1,\dots ,p_n)$ with each $p_j >0$ and $\sum_{j=1}^n p_j = 1$.  At each $p\in \mathfrak{S}_c$, 
we may identify the tangent space with the subspace of $\R^n$ consisting of vectors $k = (k_1,\dots,k_n)$ 
with 
$\sum_{j=1}^n k_j =0$.  The analogs of quantum operations are {\em stochastic  maps} $P$; i.e., elements of 
$M_n(\R)$ with non-negative and the property that for each $j$, $\sum_{i=1}^n P_{i,j} =1$.   In this setting, 
it is natural to ask for metrics $\gamma$ on  $\mathfrak{S}_c$  with the property that 
\begin{equation*}
\gamma_{Pp}(Pk,Pk) \leq \gamma_p(k,k)
\end{equation*}
for all stochastic $P$ at each $p\in \mathfrak{S}_c$ and for each tangent vector $k$. In 1982, Cencov \cite{Cen82}, building on earlier work of Fisher \cite{F25}, proved that the unique such metric, up to a constant multiple, is the {\em Fisher Information }
\begin{equation*}
\gamma_p(k,k) = \sum_{j=1}^n \frac{k_j^2}{p_j}\ .
\end{equation*}
A decade later, Cencov together with Morozova \cite{MC90}, took up the quantum problem.  In the non-commutative setting, there are many possible 
ways to ``divide by'' a non-degenerate density matrix $\rho$.   Morozova and Cencov came up with several conjectures, all eventually shown to be correct,  
that certain explicit metrics were in fact monotone metrics, although they did not resolve any of these conjectures.  An account of their work can be found in \cite{P96}. 

However, had they known of Theorems 2 and 3 of Lieb's paper, and then recognized them as monotonicity theorems,  they would have had positive solutions to the most important of their conjectures:   These theorems, written in the monotonicity forms \eqref{lieb22} and \eqref{lieb23} show that  for each $0 < t < 1$
\begin{equation}\label{che4}
\gamma_\rho^{(t)}(K,K)  :=  \tr[K \rho^{t-1} K \rho^{-t}]
\end{equation}
and
\begin{equation}\label{che5}
\widehat{\gamma}_\rho(K,K)  :=  \tr\left [\int_0^\infty K \frac{1}{s+\rho} K \frac{1}{s + \rho}{\rm d}s\right] 
\end{equation}
are monotone metrics. Note that the right hand sides are always positive for non-zero $K$, so that these do define  Riemannian metrics, and then the monotonicity is provided by \eqref{lieb22} and \eqref{lieb23}. Note that when $\rho$ and $K$ commute, we have
$$
\gamma_\rho^{(t)}(K,K)    =  \widehat{\gamma}_\rho(K,K)  = \tr\left[\frac{K^2}{\rho}\right]\ ,
$$
as one might expect. {\em Thus the functionals studied by Lieb in 1973 are both natural quantum generalizations of the Fisher Information metric}.

However, none of the early writers on the subject made the connection with Lieb's theorems, and his paper \cite{L73} is not cited in \cite{P96} which gives the first explicit proof of the fact that \eqref{che4} and \eqref{che5} do in fact define monotone metrics as had been conjectured by Morozova and Cencov. The fact that \eqref{che5} is a monotone metric plays an important role in the author's work \cite{CM17} with Jann Maas on a geometric approach to inequalities for quantum Markov semigroups with detailed balance. In our paper, we do observe that the metric in \eqref{che5} is monotone  is a consequence of Lieb's Theorem 3 in \cite{L73}, while we also give a reference to \cite{P96} for the explicit monotonicity statement. In pointing out the lack of early references to \cite{L73} by people working on monotone metrics, the intention is not to criticize, but rather to emphasize how prescient  Lieb was in pointing out  both the validity and equivalence of  first three theorems in the  1973 paper.

The right hand side of \eqref{che5} is well known to be the negative of the Hessian of the entropy $S(\rho)$. That is,
$$
\frac{\partial^2}{\partial s \partial t} \tr[ (\rho + s K + t K) \log (\rho  + s K + t K) ]\bigg|_{s=0,t=0}  =  \tr\left [\int_0^\infty K \frac{1}{s+\rho} K \frac{1}{s + \rho}{\rm d}s\right] \ ,
$$
as one can verify using the integral representation for the logarithm \eqref{logrep} that is used below. Lesniewski and Ruskai \cite{LR99} showed that all monotone metrics arise in this manner from one of the {\em quasi entropies} that had been introduced by Petz \cite{P85,P86}.

\section{The Lieb-Ruskai Monotonicity Theorem}

The Lieb-Ruskai Monotonicity Theorem is the main result of \cite{LR74}, and it asserts the monotonicity of an operator function in two variables under completely positive maps. However it is not stated in \cite{LR74} as a theorem, 
but only discussed as an example, after Theorem 2 of that paper, and with an extraneous hypothesis, namely that the completely positive map in question be unital.  Their argument does not require this at all, and hence their result, stated for matrices, is:

\begin{thm}[Lieb and Ruskai 1974]\label{LiRu}  Let $\Phi: M_m(\C) \to M_m(\C)$ be  completely positive. Then for all $A,B \in M_n(\C)$,
\begin{equation*}
\Phi(A^*B) \frac{1}{\Phi(B^*B)} \Phi(B^*A) := \lim_{\epsilon \downarrow 0} \Phi(A^*B) \frac{1}{\Phi(B^*B)+\epsilon I } \Phi(B^*A)\in M_m(\C)\ ,
\end{equation*}
and
\begin{equation}\label{LR3A}
\Phi(A^*A)  \geq  \Phi(A^*B) \frac{1}{\Phi(B^*B)} \Phi(B^*A)\ .
\end{equation}
\end{thm}

Some years later, Choi  proved:
\begin{thm}[Choi 1980]\label{LRC} The Lieb-Ruskai Monotonicity Theorem is valid for all $2$ positive maps $\Phi: M_m(\C) \to M_m(\C)$.
\end{thm}

This is an extension of Theorem~\ref{LiRu} since every completely positive map is $2$ positive, but the converse is not true; see, e.g., \cite{Choi72}.  Choi's proof uses the following well known lemma (He refers to \cite{S59} for a more general result):

\begin{lm}\label{smlm} For $X\in M_n^{++}(\C)$ and $Y,Z\in M_n(\C)$,
\begin{equation}\label{smul}
\left[ \begin{array}{cc} X & Z\\Z^* & Y\end{array}\right] \geq 0 \quad \iff \quad Y \geq Z^*X^{-1}Z\ .
\end{equation}
\end{lm}

\begin{proof} For any $v\in \C^n$,
$$\left\langle \left(\begin{array}{c} -X^{-1}Zv\\ v\end{array}\right)   \left[ \begin{array}{cc} X & Z\\Z^*& Y\end{array}\right]  \left(\begin{array}{c} -X^{-1}Zv\\ v\end{array}\right)\right \rangle = \langle v, Yv\rangle - \langle v,Z^*X^{-1}Zv\rangle\ .$$
On the other hand,
$$
\left[ \begin{array}{cc} X & Z\\Z^* & Z^*X^{-1}Z\end{array}\right] = \left[ \begin{array}{cc} X^{1/2}   & 0\\ Z^*X^{-1/2} & 0\end{array}\right] 
\left[ \begin{array}{cc} X^{1/2}   & X^{-1/2}Z\\ 0 & 0\end{array}\right] \geq 0\ .
$$

\end{proof}

\begin{proof}[Proof of Theorem~\ref{LRC}]  For all $A,B\in M_n(\C)$,
$$
\left[ \begin{array}{cc} B^*B   & B^*A \\ A^*B  & A^*A\end{array}\right] = \left[ \begin{array}{cc} B^*   & 0 \\ 0 & A^*\end{array}\right]
\left[ \begin{array}{cc} B   & A \\ 0 & 0\end{array}\right] \geq 0
\ ,$$
and if $\Phi$ is $2$-positive, then for all $\epsilon>0$,
$\left[ \begin{array}{cc} \Phi(B^*B) + \epsilon I   & \Phi(B^*A) \\ \Phi(A^*B)  & \Phi(A^*A)\end{array}\right] \geq 0$. By Lemma~\ref{smlm},
$$
\Phi(A^*A) \geq \Phi(A^*B) ( \Phi(B^*B) + \epsilon I) )^{-1} \Phi(B^*A) \ .
$$
\end{proof}

\begin{cl}\label{cl2}  Let $\Phi: M_n(\C) \to M_m(\C)$  be $2$-positive. 
Then for all $C\in M_n^{++}(\C)$ and all $X \in M_n(\C)$
\begin{equation}\label{LR7}
\Phi\left(X^*\frac{1}{C}X\right) \geq \Phi(X)^*\frac{1}{\Phi(C)}\Phi(X)\ .
\end{equation}
\end{cl}

\begin{proof} Let $B = C^{1/2}$ and let $A = C^{-1/2}X$. Then \eqref{LR3A} becomes \eqref{LR7}.
\end{proof}

\begin{exam}[Kiefer's Inequality] For example, take $\cH = \C^n$ so that $\cB(\cH)$ can be identified with $M_n(\C)$. 
Let $\cK$ the direct sum of $m$ copies of $\cH$, so that we may identify $\mathcal{A} := \cB(\cK)$ with the $m\times m$ block matrices 
whose entries are in $M_n(\C)$.   The partial trace map $\Psi_m^\dagger:\cB(\cK) \to \cB(\cH)$ defined in \eqref{contomon1BI}
is, as we have seen, completely positive and trace preserving. 

Consider any sets
 $\{X_1,\dots,X_m\}$ of $n\times n \subset M_n(\C)$  and  $\{C_1,\dots,C_m\}\subset M_n^{++}(\C)$.
 Define 
 $$
 X := \left[\begin{array}{ccc} X_1 &  &\\  &  \ddots & \\ &  & X_m\end{array}\right]  \qquad{\rm and}\qquad 
 C:= \left[\begin{array}{cccc} C_1  & &\\  &  \ddots & \\ &  & C_m\end{array}\right]
 $$
 Then  with $\Phi$ taken to be the partial trace, \eqref{LR7} becomes

\begin{equation}\label{kief}
\sum_{j=1}^m X_j^*\frac{1}{C_j} X_j \geq  \left(\sum_{j=1}^m X_j\right)^*\frac{1}{\sum_{j=1}^m C_j} \left(\sum_{j=1}^m X_j\right)  
\end{equation}
This is due to Kiefer \cite{K59}, though it was also noted by Lieb and Ruskai, who were unaware of Kiefer's work.  However, it is a very special case of their main result. 
\end{exam}

\begin{exam}[The Kadison-Choi Schwarz Inequality]\label{KCS} Taking $B = I$ and supposing that $\Phi$ is unital and $2$-positive;   \eqref{LR3A} becomes 
\begin{equation}\label{LR7K}
\Phi\left(A^*A\right) \geq \Phi(A)^*\Phi(A)\ .
\end{equation}
for all $X\in \cA$.   This was proved by Kadison for self-adjoint $A$  \cite{K52}, and in general by Choi \cite{Choi74}. 
\end{exam}

In fact, Kadison proved his inequality for $A$ self-adjoint, but only assumed that $\Phi$ was positive, not even $2$-positive. However, 
Theorem 4 of Stinespring's fundamental paper \cite{S55} states that if $\Phi: \mathcal{A} \to \cB(\cH)$ is a positive map on any commutative 
$C^*$ algebra $\mathcal{A}$, then $\Phi$ is completely positive. If $A$ is self-adjoint, or even normal, take $\mathcal{C}(A)$ to be the 
commutative $C^*$ algebra  that is the norm closure of all polynomials in $A$. Thus, Corollary~\ref{cl2} and Theorem 4 of \cite{S55} yield a short proof of 
Kadison's inequality. (Stinespring's proof of Theorem 4 is short and elementary). St\o rmer has proved a complementary result: A positive maps from one $C^*$ algebra into another commutative $C^*$ algebra is completely positive. Both theorems are discussed in \cite{Sto10}.

Unital maps that satisfy \eqref{LR7K} are known as {\em Schwarz maps}. 
(The terminology is nearly, but not completely, standard.  Petz \cite[p. 62]{P86} calls any map satisfying \eqref{LR7K} a Schwarz map.)  
One might hope that Corollary~\ref{cl2} would be valid whenever $\Phi$ is a Schwarz map. Choi proved \cite[Proposition 4.1]{Choi80} 
that the inequality \eqref{cl2} essentially characterizes the set of $2$-positive  maps: Under the condition that $\Phi(I)$ is invertible,
$\Phi: M_n(\C) \to M_m(\C)$ is $2$-positive if and only if  \eqref{cl2}  is valid for all $C\in M_n^{++}(\C)$ and all $X\in M_n(\C)$. 
He also  showed in \cite[Appendix A]{Choi80}, that there exist Schwarz maps that are not $2$-positive; e.g., the map
map $\Phi$ on $M_2(\C)$ given by
\begin{equation}\label{choimap}
\Phi(X) = \frac12 X^T  + \frac14 \tr[X] I\ ,
\end{equation}
where $X^T$ is the transpose of $X$. His construction was further developed in \cite{T85} and \cite[Example 3.6]{HMPB}.

Corollary~\ref{cl2} is the source of many monotonicity inequalities, and it will be the only inequality used in Section 12 when 
we prove, following \cite{HP12}, the monotonicity versions of Lieb's first three theorems.
However, in that proof we will need  Corollary~\ref{cl2} only in the {\em tracial} form:
\begin{equation}\label{trform}
\tr\left[\Phi\left(X^*\frac{1}{C}X\right)\right] \geq \tr\left[\Phi(X)^*\frac{1}{\Phi(C)}\Phi(X)\right]\ .
\end{equation}
One may expect  that \eqref{trform} is valid for a wider class of maps $\Phi$ than $2$-positive maps, and this is the case. For a recent result in this direction, see \cite{CMH}.

We close this section with a further extension of the Lieb-Ruskai Monotonicity Theorem. Let 
$\Phi:M_n(\C) \to M_m(\C)$ be a non-zero  positive map. As noted earlier, even if $X\in M_n^{++}(\C)$, $\Phi(X)$ need not be invertible. The map $\Xi^\dagger_{m,n}$ for $m>n$ is completely positive and trace preserving even, but $\Xi^\dagger_{m,n}(X)$ is not invertible for any $X\in M_n^{++}(\C)$. (See the second paragraph of Section 3).

The following was proved by Ando \cite[Corollary 3.1]{A79} who assumed both $H$ and $X$ to be positive definite, and using a more restrictive definition of positivity  that entailed the invertibility of $\Phi(X)$. Choi observed \cite[Proposition 4.3]{Choi80} that the same proof allows for $H$ to be merely self-adjoint.

\begin{thm}\label{ACM}  (Ando, Choi) Let 
$\Phi:M_n(\C) \to M_m(\C)$ be a  positive map such that $\Phi(I)$ is invertible.  Then for all self-adjoint $H\in M_n(\C)$ and all $X\in M_n^{++}(\C)$, 
\begin{equation}\label{ACM1}
\Phi(H X^{-1}H) \geq \Phi(H) \Phi(X)^{-1} \Phi(H)\ . 
\end{equation}
\end{thm}

\begin{proof} Define $\Psi: M_n(\C)\to M_m(\C)$ by
\begin{equation}\label{ACM2}
\Psi(Y) = \Phi(X)^{-1/2} \Phi(X^{1/2}YX^{1/2})\Phi(X)^{-1/2}\ ,
\end{equation}
and note the $\Psi$ is positive an unital. Then taking $Y := X^{-1}HX^{-1/2}$, which is self-adjoint, Kadison's inequality yields $\Psi(Y^2) \geq \Psi(Y)^2$, and this is equivalent to \eqref{ACM1}. 
\end{proof}

The condition that $\Phi(I)$ be invertible  can be relaxed as follows:  If $\Phi(I)$ is not invertible, then  $\Phi(I)$ has the block matrix structure
\begin{equation}\label{extLRM1}
\Phi(I) = U^*\left[ \begin{array}{cc} B & 0\\ 0 & 0\end{array}\right]U
\end{equation}
where $U\in M_m(\C)$ is unitary and for some $1 \leq \ell < m$, $B\in M_\ell^{++}(\C)$.   If $X\in M_n^{++}(\C)$, there are numbers $0 < c_1,c_2$ such that $c_1 I \leq X \leq c_2 I$ and hence $c_1\Phi(I) \leq \Phi(X) \leq c_2 \Phi(I)$. It follows that 
\begin{equation}\label{extLRM2}
\Phi(X) = U^*\left[ \begin{array}{cc} \widehat{X} & 0\\ 0 & 0\end{array}\right]U
\end{equation}
where $U$ is the same unitary figuring in \eqref{extLRM1} and $\widehat{X} \in M_\ell^{++}(\C)$, again with the same $\ell$ as in \eqref{extLRM1}. 
The {\em Moore-Penrose generalized inverse} of $\Phi(X)$, denoted $\Phi(X)^+$ is given by 
\begin{equation}\label{extLRM2b}
\Phi(X)^+ = U^*\left[ \begin{array}{cc} \widehat{X}^{-1} & 0\\ 0 & 0\end{array}\right]U\ .
\end{equation}

Every self adjoint $H\in M_n(\C)$ can be written as the difference of two matrices in $M_n^{++}(\C)$, and it follows that $\Phi(H)$ has the form
\begin{equation}\label{extLRM3}
\Phi(H) = U^*\left[ \begin{array}{cc} \widehat{H} & 0\\ 0 & 0\end{array}\right]U
\end{equation}
with $U$ and $\ell$ as before and $\widehat{H}\in M_\ell(\C)$.  Given such a positive map $\Phi$ we may define $\widehat{\Phi}: M_n(\C) \to M_\ell(\C)$ by
\begin{equation}\label{ACM2}
\widehat{\Phi}(Z) = \Xi_{m,\ell}(U \Phi(Z) U^*) 
\end{equation}
for all $Z\in M_n(\C)$, where, as in Section 3, $ \Xi_{m,\ell}$ ``picks off'' the upper-left $\ell\times \ell$ block of the matrix in $M_m(\C)$ to which it is applied. 
For seld-adjoint $H$, we may then recover $\Phi(H)$ from $\widehat{\Phi}(H)$ through
\begin{equation}\label{ACM3}
\Phi(H) = U^* \Xi^\dagger_{m,\ell}(\widehat{\Phi}(H))U\ .
\end{equation}

Since $\widehat{\Phi}(I)$ is invertible, we have the following corollary of Theorem~\ref{ACM}:

\begin{cl}\label{ACMcl} Let 
$\Phi:M_n(\C) \to M_m(\C)$ be a non-zero  positive map.  Then for all self-adjoint $H\in M_n(\C)$ and all $X\in M_n^{++}(\C)$, 
\begin{equation}\label{ACM4}
\Phi(H X^{-1}H) \geq \Phi(H) \Phi(X)^{+} \Phi(H)\ . 
\end{equation}
\end{cl}

\begin{proof}  With $\widehat{\Phi}$ defined as in \eqref{ACM2}, we may apply Theorem~\ref{ACM} to conclude
$$
\widehat{\Phi}(H X^{-1}H) \geq \widehat{\Phi}(H) \widehat{\Phi}(X)^{-1} \widehat{\Phi}(H)
$$
The claim then follows from \eqref{ACM3} together with the fact that 
\begin{equation}\label{ACM5}
W \mapsto  U^* \Xi^\dagger_{m,\ell}(W)U
\end{equation} is a $*$-homomorphism from $M_\ell(\C)$ into $M_m(\C)$, together with \eqref{extLRM2b}.
\end{proof}

\section{The GNS construction for $M_n(\C)$}

Uhlmann made an important contribution to our subject  \cite{Uh77} by bringing  the Gelfand-Naimark-Segal representation (GNS) of 
$M_n(\C)$ associated to the trace into the tool-kit. 
The GNS representation is explicitly mentioned at the end of Section 5 in \cite{Uh77} 
where he explains that a construction he has been using in these terms. 
Simon \cite{S05} made a simpler and more direct development of this point of view, and used it to prove the Lieb Concavity Theorem, 
and Donald \cite{Do86} had used it to prove the joint convexity of the relative entropy. It also plays a role in the work of 
Pusz and Woronowicz \cite{PW78} who cite Uhlmann \cite{Uh77}.

Regard  
$M_n(\C)$ equipped with the Hilbert-Schmidt inner product $\langle X,Y\rangle := \tr [X^*Y]$ as a Hilbert space.  For $X\in M_n(\C)$, 
define $L_X$ to be the operator on $\HH$ given by $L_XA := XA$. That is, $L_X$ is left multiplication by $X$. It is readily checked that 
$X \mapsto L_X$ is a $*$-homomorphism from $M_n(\C)$ into the linear operators on $M_n(C)$ regarded as a Hilbert space, 
and this is the GNS representation of $M_n(\C)$ induced by the trace.

 Likewise, define $R_X$ to be the operator of right multiplication by $X$; that is, $R_XB = BX$.   Notice that 
$\langle C, R_X B\rangle  = \tr[C^*BX] = \tr[(CX^*)^*B] = \langle R_{X^*}C,B\rangle$, and hence $(R_X)^\dagger = R_{X^*}$. A similar computation shows that 
$(L_X)^\dagger = L_{X^*}$. In particular, if $X$ is self-adjoint so are both $R_X$ and $L_X$, and then it is easy to see that if 
$X$ is strictly positive so are $R_X$ and $L_X$, 
and indeed when $X$ is self-adjoint, $X$, $L_X$ and $R_X$ all have the same spectrum, and for any real valued  function $f$ defined on the spectrum of $X$, 
$f(L_X) = L_{f(X)}$ and $f(R_X) = R_{f(X)}$.  Since $R_X$ and  $L_Y^{-1}$ are positive and commute, $R_XL_Y^{-1}$ is positive, and hence for any 
$f:(0,\infty) \to (0,\infty)$, $f(R_XL_Y^{-1})$ is well defined by the Spectral Theorem. 

For example, consider the function $f(x) = x \log x -x +1$. It is evident that $f(x) \geq 0$ with equality only at $x= 0$  Hence $X,Y\in M_n^{++}(\C)$,
$f(R_XL_Y^{-1})$ is a positive operator on $M_n(\C)$, and then so is 
$$f(R_XL_Y^{-1})L_Y  = R_X (\log R_x - \log L_Y)  - R_X + L_Y\ .$$ 
But then
$$0 \leq \langle I, f(R_XL_Y^{-1})L_Y I\rangle  = \tr[X(\log X- \log Y)] -\tr[X] +\tr[Y]\ .$$
If $X$ and $Y$ are density matrices, then the right side is $D(X||Y)$, and this shows that for density matrices, $D(X||Y) \geq 0$, and in fact, expanding $I$ in a basis $\{|u_i\rangle \langle v_j|\}_{1 \leq i,j\leq n}$ where the $u_i$'s are eigenvectors of $Y$ and the $v_j$'s are eigenvectors of $X$,  one see that
$D(X||Y) = 0$ if and only if $X= Y$. This provides the  second  elementary proof of \eqref{klein} promised in Section 2, this time without using Klein's inequality. 
The second proof of \eqref{klein} provides only the barest hint of the utility of writing trace functionals in terms of the  GNS representation.  To get the most out of this approach, we need one more tool, namely integral representations for operator monotone and convex functions.

\section{Operator monotonicity and convexity}

A function $f:(0,\infty)\to \R$ is said to be {\em operator monotone increasing} in case for all   $A,B\in M_n^{++}(\C)$, any $n$, 
$A\geq B$ implies $f(A) \geq f(B)$, and $f$ is said to  {\em operator monotone decreasing} if $-f$ is operator monotone increasing. 

Consider  $f(x) = x^{-1}$. For $A,B> 0$, define $X := (A+B)^{-1/2}A(A+B)^{-1/2}$  and $Y :=  (A+B)^{-1/2}B(A+B)^{-1/2}$. 
Evidently, $A\geq B \iff X \geq Y$, and $Y^{-1} \geq X^{-1} \iff B^{-1} \geq A^{-1}$. Since $X+Y =1$, $X$ and $Y$ commute, and then by the spectral theorem $X \geq Y$ implies $Y^{-1}\ \geq X^{-1}$ simply because $f(x) = x^{-1}$ is monotone decreasing in $x$. Thus, the function $f(x) = x^{-1}$ is operator monotone decreasing. This one example leads to many others: For instance, let $0 < t < 1$.  Then there is the integral representation
$$x^t  =  \frac {\sin(\pi t)} {\pi} \int_0^\infty \lambda^t \left(\frac{1}{\lambda} - \frac{1}{\lambda+x}\right){\rm d}{\lambda}\ ,$$
and now it follows from what we have just proved that $f(x) = x^t$ is operator monotone increasing for all $0 < t < 1$. By a theorem of L\"owner \cite{Lo34}, every operator monotone increasing functions has an integral representation of this general form (see below), and this is the deep part of the theory. Simon's book \cite{S19} contains a beautiful account, with many proofs, some new, of L\"owner's Theorem. However, for the specific examples that arise in this paper, elementary arguments suffice, such as the ones provided just above.  For the generalizations that follow, more is required.

We quote L\"owner's Theorem \cite{Lo34} as stated  in Ando and Hiai \cite{AH11}; see also \cite{A78} and \cite{S19}:

\begin{thm}[L\"owner's theorem]\label{Lot} For $x,\lambda \in [0,\infty)$, define the function $\phi(x,\lambda)$ by
\begin{equation}\label{low1}
\phi(x,\lambda) := (1 + \lambda) \frac{x}{\lambda + x}\ ,
\end{equation}
and define $\phi(x,\infty) := x = \lim_{\lambda\uparrow\infty}\phi(\lambda,x)$,
and notice that for each $x$, $\phi(x,\lambda)$ is a bounded function of $\lambda$, so that for any finite positive Borel measure $\mu$  on $(0,\infty)$, all $\beta \in \R$ and all $\gamma \geq 0$,
\begin{equation}\label{low2}
h(x) :=  \beta + \gamma x + \int_{(0,\infty]} \phi(x,\lambda){\rm d}\mu\ ,
\end{equation}
is a well defined function on $\R_+$.  The mapping $(\beta,\gamma,\mu) \mapsto h$ is an affine isomorphism  onto the class of operator monotone increasing functions.
\end{thm}

Note that 
\begin{equation}\label{lowid}
 \frac{x}{\lambda + x}   =  1 -\frac{\lambda}{\lambda+x} \ ,
\end{equation}
from which it is clear that for each $\lambda$, $\phi(\lambda, x)$ is concave and monotone increasing in $x$, not only as a function of a real variable, but also in the operator sense.   Thus all operator monotone functions $h$  that are real valued on $[0,\infty)$ are also operator concave, and thus $-h$ is operator convex, and monotone decreasing. 

By Theorem 3.3 of Bendat and Sherman \cite[Theorem 3.3]{BS75}, a real valued function $f$ on $(0,\infty)$ is operator convex if and only if for each $x_0\in (0,\infty)$,
\begin{equation*}
h(x) := \frac{f(x) - f(x_0)}{x- x_0}
\end{equation*}
is an operator monotone increasing function. An operator convex function is necessarily convex in the ordinary sense, and hence $f$ has a right deriviative $f'_+(x)$ at each $x>0$, and this is an increasing function of $x$. Therefore 
$f'_+(0) :=\lim_{x\downarrow 0}f'_+(x)$ exists, though it may take the value $-\infty$. 
It may be deduced from the result of Bendata and sherman  that every operator convex function $f$ on $[0,\infty)$ with $f'_+(0) \in \R$ has an integral representation of the form
\begin{equation}\label{irep0}
f(x) = \alpha + \beta x + \gamma x^2 + \int_{(0,\infty)}(\lambda+1) \frac {x^2}{\lambda+ x} {\rm d}\mu
\end{equation}
where $\mu$ is a finite Borel measure and where $\alpha,\beta\in \R$ and $\gamma \geq 0$.  It will often be convenient to extend the integrand to $(0,\infty]$ by continuity, and then replace $\mu$ with the finite Borel measure $\nu$ on $(0,\infty]$ that agrees with $\mu$ on $(0,\infty)$ and with $\nu(\{\infty\}) = \gamma$. Then we can rewrite \eqref{irep0} as
\begin{equation}\label{irep1}
f(x) = \alpha + \beta x +  \int_{(0,\infty]}(\lambda+1) \frac {x^2}{\lambda+ x} {\rm d}\nu
\end{equation}

The condition $f'_+(0)> -\infty $ excludes certain cases such as $f(x) = -\log x$ and $f(x) = x \log x$ that come up in the next sections, but these cases are dealt with by using the elementary integral representation for the logarithm given below  in \eqref{logrep}. Alternatively, one has the integral representation 
\eqref{irep0} for $f(x+\epsilon)$  for any $\epsilon>0$.

Recall that if $f$ is operator monotone increasing, then $f$ is operator concave, and hence if $f$ is operator monotone decreasing, then $f$ is operator convex. For this special class of operator convex functions, there is another integral representation due to Hansen \cite{H06}:  

\begin{lm}[Hansen]\label{hanrep} A function $f(x) : (0,\infty) \to (0,\infty)$ is monotone decreasing if and only if 
\begin{equation}\label{han1}
f(x) = \beta + \int_{[0,\infty)}\frac{\lambda+1}{\lambda+x}{\rm d}{\nu}
\end{equation}
where ${\nu}$ is a finite Borel measure on $[0,\infty)$ and $\beta \geq 0$.
\end{lm}

For the reader's convenience, here is the short proof of this due to Ando and Hiai \cite{AH11}.  

\begin{proof}
Note that $f(x)$ is operator monotone decreasing if and only if $f(1/x)$ is operator monotone increasing, and hence
$$
f(1/x) = \beta + \int_{(0,\infty]}(\lambda+1)\frac{x}{\lambda+x}{\rm d}\mu
$$
for some  finite Borel measure $\mu$ on $(0,\infty)$, and some $\beta\in \R$. 

Define ${\nu}$ to the be the push-forward of ${\mu}$ under the map $\lambda \mapsto \lambda^{-1}$.
Then
\begin{eqnarray*}
f(x) &=&  \beta +  \int_{(0,\infty]}(\lambda+1)\frac{1/x}{\lambda+1/x}{\rm d}\mu  =   \beta +  \int_{(0,\infty]}(\lambda+1)\frac{1}{\lambda x+ 1}{\rm d}\mu\\
&=&  \beta +   \int_{[0,\infty)}(1/\lambda+1)\frac{1}{x/\lambda + 1}{\rm d}{\nu} =  \beta +   \int_{[0,\infty)}\frac{\lambda+1}{\lambda+x}{\rm d}{\nu}\ .\\
\end{eqnarray*}
Since $\beta = \lim_{x\uparrow \infty}f(x)$, $\beta \geq 0$. Evidently \eqref{han1} defines an operator monotone decreasing function. 
\end{proof}

The following theorem is due to Hansen \cite[Proposition 4,3, Remark 4.6]{H06b} who proved the ``if'' part, and to
 Ando and Hiai \cite{AH11}  who proved the ``only if''  part.   Their proof is provided  in the context of an investigation of operator log-convexity, and it is a 
consequence of a theorem asserting  the equivalence  of 13 conditions. The following short  and direct proof of  the ``only if'' part of this theorem  was provided to me by Frank Hansen in  recent correspondence, and I am thankful to him for permission to present it here.

\begin{thm}[Hansen, Ando and Hiai]\label{AH1} Let  $f:(0,\infty) \to (0,\infty)$. Then
$$
(X,v) \mapsto \langle v, f(X)v\rangle
$$ 
is jointly convex on $M_n^{++}(\C)\times \C^n$ if and only if $f$ is operator monotone decreasing. 
\end{thm}

\begin{proof}  Suppose first  that $f(x): (0,\infty) \to (0,\infty)$ is operator monotone decreasing. 
Hansen \cite{H06b} 
showed that $(v,X) \mapsto \langle v, (\lambda+X)^{-1}v\rangle$
is jointly convex.  In \cite[Remark 4.5]{H06b} he gave a proof of this fact that 
was suggested to him by Lieb:  Choose any {\em fixed} unit vector $u$, and then for each vector $v$, 
define $Z_v$ to be the rank one operator $Z_v := |v\rangle \langle u|$. Then for  each $\lambda$, 
$$
 \langle v, (\lambda+X)^{-1}v\rangle  = \tr[Z_v^* (\lambda+X)^{-1}Z_v]\ ,
$$
and now the joint convexity  of $(v,X) \mapsto \langle v, (\lambda+X)^{-1}v\rangle$ follows from Kiefer's Inequality \eqref{kief}, and  the 
joint convexity  of $(v,X) \mapsto \langle v, f(X)v\rangle$ then follows from the integral representation  \eqref{han1}.

Suppose next that  $(X,v) \mapsto \langle v, f(X)v\rangle$ is jointly convex. Define $g(x) := x^2 f(x)$. 
Then for $X,Y\in M_n^{++}(\C)$ and $0< s  < 1$, and any $v\in \C^n$
\begin{eqnarray*}  \langle v, g((1-s)X + s Y)v\rangle &=& \langle ((1-s)X + s Y)v, f((1-s)X + s Y) ((1-s)X + s Y)v\rangle\\  &\leq& (1-s)\langle Xv, f(X) Xv\rangle  + s\langle Yv, f(Y) Yv\rangle\\  &=& (1-s)\langle v,g(X)v\rangle  + s \langle v,g(Y)v\rangle\ .
\end{eqnarray*}
Thus, $g$ is operator convex, and evidently $f$ is operator convex.  It is not hard to see that $g'_+(0) > -\infty$, but one may replace $f(x)$ by $f(x+\epsilon)$ for $\epsilon>0$, and then take $\epsilon$ to $0$ at the end. Either way, $g$ has an integral representation of the form
\begin{equation*}
g(x) = \alpha + \beta x + \gamma x^2 + \int_{(0,\infty)}(\lambda+1) \frac {1}{\lambda+ x} {\rm d}\mu\ ,
\end{equation*}
where $\mu$ is a finite Borel measure on $(0,\infty)$. Therefore
$$
f(x) = \alpha x^{-2} + \beta x ^{-1}+ \gamma  + \int_{(0,\infty)}(\lambda+1) \frac {1}{\lambda+ x} {\rm d}\mu\ ,
$$
and since $f$ is operator convex and non-negative $\alpha =0$ and $\beta,\gamma \geq 0$.    It follows immediately that $f$ is operator monotone decreasing. 
\end{proof}

\section{Inequalities relating to operator monotonicity}

The following fundamental theorem of Hiai and Petz \cite{HP12} is a substantial generalization of the first three theorems in \cite{L73}, and we explain in the example following the theorem how one recovers those theorems by making two special choices of the  function $f$ figuring in the theorem, namely $f(x) = x^t$ and $f(x) = \int_0^1 x^t {\rm d} t$. The proof is extremely simple, and the only inequality that is used is the Lieb-Ruskai Monotonicity  Theorem. 

\begin{thm}[Hiai and Petz]\label{opmontI}  Let $f:(0,\infty) \to (0,\infty)$ be continuous, and let $\Phi$ be a unital and $2$-positive map from $M_n(\C)$ to $M_m(\C)$. For $X,Y\in M_n^{++}(\C)$,  
define $$G_f(X,Y) = f(R_X L_Y^{-1})L_Y\ ,$$
and note that this is a positive definite operator.  The following are equivalent:

\smallskip
\noindent{\it (1)}  The function $f$ is operator monotone increasing. 

\smallskip
\noindent{\it (2)}  For all positive definite $X,Y\in M_n(\C)$, 
\begin{equation}\label{ph2Q}
\Phi G_f(\Phi^\dagger(X), \Phi^\dagger(Y))^{-1}\Phi^\dagger     \leq G_f(X,Y)^{-1}\ .
\end{equation}

\smallskip
\noindent{\it (3)}  For all positive definite $X,Y\in M_n(\C)$,  
\begin{equation}\label{ph1Q}
\Phi^\dagger  G_f(X,Y)  \Phi  \leq   G_f(\Phi^\dagger(X), \Phi^\dagger(Y))\ .
\end{equation}

\smallskip
\noindent{\it (4)}  The map 
\begin{equation}\label{ph3Q}
(X,Y,Z)\mapsto  \langle Z, G_f(X,Y)^{-1} Z\rangle
\end{equation} is jointly convex on $M_n^{++}(\C)\times M_n^{++}(\C) \times M_n(\C)$.
\end{thm}

\begin{exam}  The first three theorems of \cite{L73} follow directly from Theorem~\ref{opmontI} and the elementary 
fact that $f(x) = x^t$, $0 \leq t \leq 1$ is operator monotone. 
Then for $X,Y > 0$, $G_f(X,Y) = R_X^tL_Y^{1-t}$.  
Then $\langle K, \Phi^\dagger  G_f(X,Y)  \Phi      K \rangle \leq \langle K, G_f(\Phi^\dagger(X), \Phi^\dagger(Y))K\rangle$  
is the same as \eqref{lieb21b}. Likewise, $\langle K, \Phi G_f(\Phi^\dagger(X),
 \Phi^\dagger(Y))^{-1}\Phi^\dagger  K\rangle \leq \langle K, G_f(X,Y)^{-1}K\rangle$ is the same as \eqref{lieb22}.  

Next, define $f(x) = \int_0^1 x^t{\rm d}t = (x-1)/ \log x$, so that $f$ is operator monotone increasing. 
Then for $X,Y > 0$, $$G_f(X,Y)K = \int_0^1 Y^{1-t}KX^t{\rm d}t\ ,$$ and then a simple computation gives
$$G_f(X,Y)^{-1}K  = \int_0^\infty \frac{1}{sI + Y} K \frac{1}{sI + X}{\rm d}s \ .$$
One way to do this is to observe that it suffices to consider $K$ of the form 
$|v\rangle\langle u|$ where $v$ is an eigenvector of $Y$ and $u$ is an eigenvector of $X$. 
Therefore
by {\it (2)} of Theorem~\ref{opmontI},  for all unital $2$-positive $\Phi$, \eqref{lieb23} is valid. {\em Note that this gives somewhat more than the Theorems~\ref{L1M}, \ref{L2M} and \ref{L3M}}:  Here we have only had to assume that $\Phi$ is $2$-positive, not completely positive. 
\end{exam}

We now give the impressively simple proof of the Hiai-Petz Theorem, beginning with a simple lemma that is abstracted from their paper \cite{HP12}.

\begin{lm}\label{flipG}  For positive invertible $B\in M_m(\C)$ and $C \in M_n(\C)$ and any $m\times n$ matrix $A$,
\begin{equation}\label{ph3}
A^* B^{-1}A \leq C^{-1} \quad \iff \quad ACA^* \leq B\ ,
\end{equation}
\end{lm}

\begin{proof}  Evidently $A^* B^{-1}A \leq C^{-1} \iff C^{1/2}A^*B^{-1}AC^{1/2} \leq I$. However,
$C^{1/2}A^*B^{-1}AC^{1/2}  =  (B^{-1/2}AC^{-1/2})^* (B^{-1/2}AC^{-1/2})$ has the same non-zero spectrum as  
$$ (B^{-1/2}AC^{-1/2})(B^{-1/2}AC^{-1/2})^*  = B^{-1/2}ACA^*B^{-1/2}\ ,$$
and hence  $C^{1/2}A^*B^{-1}AC^{1/2} \leq I \iff B^{-1/2}ACA^*B^{-1/2} \leq I$, which yields \eqref{ph3}.
\end{proof}

\begin{cl}\label{flip}
Let $f:(0,\infty) \to (0,\infty)$ be a continuous function. For positive definite $X,Y\in M_n(\C)$,  define $G_f(X,Y) = f(R_X L_Y^{-1})L_Y$ , which is then an invertible operator on $M_n(\C)$. Let $\Phi:M_m(\C) \to M_n(\C)$ be unital and $2$-positive.   Then  
\begin{equation}\label{ph1}
\Phi G_f(\Phi^\dagger(X), \Phi^\dagger(Y))^{-1}\Phi^\dagger     \leq G_f(X,Y)^{-1}
\end{equation}
if and only if
\begin{equation}\label{ph2}
\Phi^\dagger  G_f(X,Y)  \Phi  \leq   G_f(\Phi^\dagger(X), \Phi^\dagger(Y))\ .
\end{equation}
\end{cl} 

\begin{proof}  Apply Lemma~\ref{flipG} with $B =G_f(\Phi^\dagger(X), \Phi^\dagger(Y))$, $C = G_f(X,Y)$ and $A = \Phi^\dagger$.
\end{proof}

\begin{proof}[Proof of Theorem~\ref{opmontI}]   We first show that {\it(1)} $\Rightarrow$ {\it(3)}. That is, suppose that $f$ is operator monotone increasing. We will then show that  \eqref{ph2} 
is valid. 
By \eqref{low1} and \eqref{low2}
$$f(x) :=  \beta + \gamma x + \int_{(0,\infty]}   \frac{x}{\lambda + x}  (1 + \lambda)   {\rm d}\mu\ ,$$
with $\beta,\gamma \geq 0$. 
Hence
${\displaystyle
G_f(X,Y) =  \beta  L_Y  + \gamma R_X + \int_{(0,\infty]}   \frac{R_X}{\lambda + R_XL_Y^{-1}}  (1 + \lambda)   {\rm d}\mu}$.
Hence to prove \eqref{ph2}, it suffices to prove that
\begin{equation}\label{ph7}
\Phi^\dagger L_Y \Phi \leq L_{\Phi^\dagger(Y)}\ ,\qquad   \Phi^\dagger R_X \Phi \leq R_{\Phi^\dagger(X)}
\end{equation}
and that
\begin{equation}\label{ph8}
\Phi^\dagger   \frac{R_X}{\lambda + R_XL_Y^{-1}}  \Phi  \leq  \frac{R_{\Phi^\dagger(X)}}{\lambda + R_{\Phi^\dagger(X)}L_{\Phi^\dagger(Y)}^{-1}} \ .
\end{equation}
For any $K$,
\begin{eqnarray*}
\langle K, \Phi^\dagger L_Y \Phi K\rangle &=& \tr[\Phi(K^*) Y \Phi(K)] = \tr[\Phi(K)\Phi(K^*) Y ] \\
 &\leq& \tr[\Phi(KK^*)Y] = \tr[KK^*\Phi^\dagger(Y)] = \langle K, L_{\Phi^\dagger(Y)} K\rangle\ ,
\end{eqnarray*}
where we have used the Kadison-Choi Schwarz Inequality. The proof of the second inequality in \eqref{ph7} is entirely similar.

By Lemma~\ref{flip}, \eqref{ph8} is equivalent to
${\displaystyle 
\Phi \left(\frac{R_{\Phi^\dagger(X)}}{\lambda + R_{\Phi^\dagger(X)}L_{\Phi^\dagger(Y)}^{-1}}  \right)^{-1}\Phi^\dagger  \leq \left(   \frac{R_X}{\lambda + R_XL_Y^{-1}} \right)^{-1}\ ,
}$
which is the same as
${\displaystyle
\Phi  \left(  \lambda  R_{\Phi^\dagger(X)}^{-1} + L_{\Phi^\dagger(Y)}^{-1}  \right)\Phi^\dagger   \leq \lambda R_X^{-1} + L_Y^{-1}
}$.
This is true if and only if for all $n\times n$ matrices $Z$, 
\begin{multline*}
\lambda \tr[ \Phi^\dagger(Z) (\Phi^\dagger(X))^{-1} \Phi^\dagger(Z^*) ] +  \tr[ \Phi(Z^*) (\Phi^\dagger(Y))^{-1} \Phi^\dagger(Z)] \leq\\ \lambda \tr[ZX^{-1}Z^*] + \tr[Z^*Y^{-1}Z]\ .
\end{multline*}
However, by Corollary~\ref{cl2} of the Lieb-Ruskai Monotonicity Theorem (with the observation that this is still valid for $2$-positive maps),  together with the fact that $\Phi^\dagger$ is trace preserving since $\Phi$ is unital, 
$
\tr[ \Phi^\dagger(Z) (\Phi^\dagger(X))^{-1} \Phi^\dagger(Z^*) ]  \leq \tr[ZX^{-1}Z^*]$ and $\tr[ \Phi(Z^*) (\Phi^\dagger(Y))^{-1} \Phi^\dagger(Z)]  \leq  \tr[Z^*Y^{-1}Z]$.
This proves {\it (1)} $\Rightarrow$ {\it (3)} and by Lemma~\ref{flip}, {\it (2)} $\iff${\it (3)}.  
Next, {\it (2)}  $\Rightarrow$ {\it (4)}, i.e., that 
the map $(X,Y,Z)\mapsto  \langle Z, G_f(X,Y)^{-1} Z\rangle$ is jointly convex,  by taking $\Phi^\dagger$ to be the partial trace as in the proof of Kiefer's inequality using Corollary~\ref{cl2}. 

We now show that {\it (4)} $\Rightarrow$ {\it (1)}, i.e., that $f$ is operator monotone increasing.  Let $u$ be any unit vector in $\C^n$ and then for any $v\in \C^n$ define $Z = |v\rangle \langle u|$.  Note that for $Y =I$, then $G_f(X,I) = f(R_X) = R_{f(X)}$. Hence the joint convexity of $(X,Y,Z) \mapsto \langle Z, G_f(X,Y)^{-1} Z\rangle$ implies the joint convexity of the map
$(X,v) \mapsto \langle v, f^{-1}(X) v\rangle$.   By a Theorem~\ref{AH1}, this means that $1/f$ is operator monotone decreasing, and then $f$ is operator monotone increasing. 
\end{proof}

\section{Trace  inequalities associated to operator convex functions }

The Data Processing Inequality
$$
D(\Phi^\dagger(X)||\Phi^\dagger(Y))  \leq D(X||Y)
$$
lies  outside the direct scope of Theorem~\ref{opmontI} because $D(X||Y)$ cannot be written in terms of $G_f(X,Y)$ for any operator monotone function $f$ from $(0,\infty)$ to $(0,\infty)$. However, one does have $D(X||Y)= \langle I, G_f(X,Y)  I \rangle$  for $f(x)= x\log x$ and also 
$D(Y||X)= \langle I, G_f(X,Y)  I \rangle$ for $f(x) = - \log x$.  Both of these functions are operator convex,  as may be seen from the integral representation
\begin{equation}\label{logrep}
\log (x) = \int_0^\infty \left(\frac{1}{\lambda+1} - \frac{1}{\lambda +x}\right){\rm d}\lambda\ .
\end{equation}

Alternatively, in 1986 Donald \cite{Do86}  gave an entirely elementary proof of the joint convexity of $D(X||Y)$ taking as his starting point the integral formula
\begin{equation}\label{eq4}
\int_0^1 \frac{(y-x)^2 t}{(1-t)x + ty}{\rm d} t =  x(\log x - \log y) + y -x
\end{equation}
valid for all $x,y>0$. 
Applying this with $R_X$ in place of $x$ and $L_Y$ in place of $y$, where $X$ and $Y$ are two non-degenerate density matrices, he obtained
\begin{equation}\label{eq5}
D(X||Y) = \int_0^1\left\langle \one, (L_Y-R_X) \frac{1}{(1-t)R_X + t L_Y}  (L_Y-R_X) \one \right\rangle  t{\rm d}t \ .
\end{equation}
Now the joint convexity is an immediate consequence of Kiefer's inequality. However, Donald did something {\em even more elementary}: He used this formula to deduce an expression
$$
D(X||Y)  =  \sup \{  \tr[XC  + YD ]\ : \ (X,Y) \in \Omega \}\ ,
$$
where $\Omega$ is specific convex set in $M_n(\C) \times M_n(\C)$ consisting of pairs of self-adjoint matrices.    This effectively displays the relative entropy as a Legendre transform, and specifies the Legendre transform of the relative entropy. His proof  built on ideas of Pusz and Woronowicz \cite{PW75,PW78} who also  gave a proof of the Lieb Concavity Theorem and the joint convexity of the relative entropy by giving explicit Legendre transform representations in \cite[Section 4]{PW78}.  The variational formula of  Pusz and Woronowicz was rediscovered by Kosaki in 1986 \cite{Ko86}.  (Kosaki wrote that his result is implicit in \cite{PW75}, but he does not cite \cite{PW78} in which his formula is explicitly proved.)

We return to the study of operators of the form 
\begin{equation}\label{gfdef}
G_f(X,Y) := f(R_X L_Y^{-1})L_Y \ ,
\end{equation}
where now we shall take $f$ to be operator convex.   Some  useful monotonicity theorems in this setting can be obtained from  integral representations and the Lieb-Ruskai Monotony Theorem by making somewhat use of the {\em perspective function} construction \cite{DM08} which associates the function 
$$g_f(x,y) := f(x/y)y$$
 to the function $f(x)$.   
Note that for $f(x) = x^2$, the associated perspective function is $x^2/y$. The integral representation of a general operator convex function 
allows it to be expressed in functions of this type, opening the way to the application of the  Lieb-Ruskai Monotony Theorem.

\begin{thm}\label{goodrep} Let $f$ be an operator convex function on $[0,\infty)$ with $f'_+(0)> -\infty$. Let $g_f(x,y)$ be the perspective function of $f$. 
Then for some finite Borel measure $\nu$ on $[0,\infty]$, and some $\alpha,\beta \in \R$,
\begin{equation}\label{low13}
g_f(x,y) 
 =  \alpha y + \beta x  +  \int_{(0,\infty]} \left(   \frac{x^2}{x + \lambda y}   \right)  (1+\lambda) {\rm d} \nu \ .
\end{equation}
\end{thm}

\begin{proof} From the integral representation \eqref{irep1} for $f$,  and the identity
$
{\displaystyle \frac{x^2}{\lambda+x} = x -  \frac{\lambda x}{\lambda +x}}
$,
\begin{eqnarray*}
g_f(x,y) &=&  \alpha y + \beta x  + \int_{(0,\infty]} \left(  x - \lambda \frac{x}{\lambda + x/y}\right)(1+\lambda) {\rm d}\nu\nonumber\\
&=& \alpha y  + \beta x   + \int_{(0,\infty]} \left(  x - \lambda \frac{1}{(\lambda^{-1}x)^{-1} + y^{-1}}\right)(1+\lambda) {\rm d}\nu\ .
\end{eqnarray*}
Now using the identity, valid for all $a,b>0$, $(a^{-1}+b^{-1})^{-1}  = a - a(a+b)^{-1}a$,
which follows easily from $a^{-1} + b^{-1} = a^{-1}(a+b) b^{-1}$,  \eqref{low13} is proved. 
\end{proof}

The following  lemma can be found in Effros' paper \cite{E09}.

\begin{lm}\label{GNSsub}  For operator convex functions $f$ on $(0,\infty)$, let $G_f$ be feined by \eqref{gfdef}. Then for all positive invertible $X_1,X_2,Y_1,Y_2$,
\begin{equation}\label{low35}
 G_f(X_1+X_2,Y_1+Y_2)  \leq G_f(X_1,Y_1) +  G_f(X_2,Y_2)\ .
\end{equation}
and consequently,
$(X,Y) \mapsto G_f(X,Y)$
is jointly convex. 
\end{lm}

\begin{proof}  Replacing $f(x)$ by $f(x+\epsilon)$ for $\epsilon>0$ if need be, we have
the integral representation \eqref{low13}, which is equivalent to
\begin{equation}\label{low33}
g_f(x,y) =  \alpha y + \beta x +\int_{(0,\infty]} \left(  x \frac{1}{x + \lambda y} x\right)  (1+\lambda) {\rm d} \nu\ ,
\end{equation}
\begin{equation}\label{low34}
G_f(X,Y) = \alpha L_Y + \beta R_X + \int_{(0,\infty]} \left( R_X \frac{1}{R_X + \lambda L_Y}R_X\right) (1+\lambda) {\rm d} \nu\ .
\end{equation}
Then \eqref{low35} follows directly from the corollary of the Lieb-Ruskai Theorem asserting that $(A,Z) \mapsto Z^*A^{-1}Z$ is sub-additive  so that for all positive invertible $X_1,X_2,Y_1,Y_2$,
\begin{equation}\label{low35B}
G_f(X_1+X_2,Y_1+Y_2)  \leq G_f(X_1,Y_1)+ G_f(X_2,Y_2)\ ,
\end{equation}
and then since $(X,Y) \mapsto  G_f(X,Y)$ is homogeneous of degree one, it is also jointly convex. 
\end{proof}

\begin{cl} For all operator convex functions $f$ on $\R_+$, and all $K \in M_n(\C)$, the function 
\begin{equation}\label{low56}
(X,Y) \mapsto \tr[K^* f(R_XL_Y^{-1})YK]
\end{equation}
is jointly convex on $M_n^{++}(\C)$. 
\end{cl}

\begin{proof} Evidently $ \tr[K^* f(R_XL_Y^{-1})YK] = \langle K, G_f(X,Y) K\rangle$.
\end{proof} 

\begin{exam}
Taking $f(x) = -\log x$ and $K = I$, this yields the joint convexity of $$(X,Y) \mapsto  \tr[Y( \log Y - \log X)]\ .$$
\end{exam}

\begin{exam}\label{lconc}
Likewise, taking $f(x) = x^q$, $q \in (1,2]$, we obtain Ando's Convexity Theorem: For all $K\in M_n(\C)$, and all $1 \leq q \leq 2$,
\begin{equation}\label{ando1}
(X,Y) \mapsto \tr[K^* Y^{1-q} K X^{q}]
\end{equation}
is jointly convex.  More generally, Ando proved that if $1 \leq q \leq 2$, and $0 \leq p \leq q-1$,
\begin{equation}\label{ando2}
(X,Y) \mapsto \tr[K^* Y^{-p} K X^{q}]
\end{equation}
is jointly convex. However, as with the Lieb Concavity Theorem, the general case \eqref{ando2} follows from the special case 
\eqref{ando1} by an entirely analogous argument. 
\end{exam}

The next results make a somewhat different use of the perspective function construction which does not involve the GNS representation.

\begin{defi}\label{oppers}  Let $f:[0,\infty) \to [0,\infty)$.  Then for all positive invertible operators $X,Y$ on a Hilbert space $\cH$,  the {\em operator perspective function} $g_f(X,Y)$ is defined by 
\begin{equation}\label{opper}
g_f(X,Y)  := Y^{1/2} f(Y^{-1/2}X Y^{-1/2}) Y^{1/2}\ ,
\end{equation}
where $f(Y^{-1/2}X Y^{-1/2})$ is defined using the spectral theorem.  
\end{defi}

\begin{thm}\label{joop}  Let  $f$ be an operator convex function on $[0,\infty)$.  For all positive 
 invertible operators $X,Y$ on a Hilbert space $\cH$, let $g_f(X,Y)$ be defined by \eqref{opper}.   Then for all $2$-positive maps $\Phi$, 
 \begin{equation}\label{low21}
\Phi(g_f(X,Y) ) \geq   g_f(\Phi(X),\Phi(Y)) \ .
\end{equation}
\end{thm}

\begin{proof}  Suppose that  $f$ is an operator convex function on $[0,\infty)$. Then from the integral representation \eqref{irep1},
\begin{eqnarray*}
f(Y^{-1/2}X Y^{-1/2}) &=& \alpha + \beta Y^{-1/2}X Y^{-1/2}\nonumber\\
&+&  \int_{(0,\infty]} Y^{-1/2}X Y^{-1/2} \frac {1}{\lambda+ Y^{-1/2}X Y^{-1/2}} Y^{-1/2}X Y^{-1/2}{\rm d}(\lambda+1)\nu\ .
\end{eqnarray*}
It follows that
\begin{equation*}
g_f(X,Y) = \alpha Y + \beta X 
+  \int_{(0,\infty]}X\frac {1}{\lambda Y+ X } X(\lambda+1)  {\rm d}\nu \ ,
\end{equation*}
and hence
\begin{equation*}
\Phi(g_f(X,Y)) = \alpha  \Phi(Y) + \beta \Phi(X) 
+  \int_{(0,\infty]}\Phi\left(X\frac {1}{\lambda Y+ X } X\right) (\lambda+1) {\rm d} \nu \ .
\end{equation*}
By Corollary~\ref{cl2} to the Lieb-Ruskai Monotonicity Theorem, 
$$
\Phi\left(X\frac {1}{\lambda Y+ X } X\right)  \geq \Phi(X) \frac{1}{ \lambda \Phi(Y) + \Phi(X)} \Phi(X)\ .
$$
\end{proof}

\begin{remark}\label{joopR} One can dispense with the condition that $\Phi$ be $2$-positive, and instead assume only that $\Phi$ 
is a non-zero positive map. In the case $\Phi(I)$ is invertible, all one need do is to apply Theorem~\ref{ACM} in place of  
Corollary~\ref{cl2} in the proof of Theorem~\ref{joop}.   In the general case, one proceeds as in the proof of Corollary~\ref{ACMcl}:
One first defines the positive  map $\widehat{\Phi}$  in terms of $\Phi$ as in \eqref{ACM2}, and then since this map is positive with $\Phi(I)$ invertible, 
one obtains
$$
\widehat{\Phi}(g_f(X,Y) ) \geq   g_f(\widehat{\Phi}(X),\widehat{\Phi}(Y)) \ .
$$
Now one uses \eqref{ACM3} to recover $\Phi$ from $\widehat{\Phi}$, and crucially, the $*$-homomorphism property \eqref{ACM5}.   See \cite[Theorem 7.10]{HUW} for a proof of Theorem~\ref{joop} for positive maps $\Phi$ under the assumption that $\Phi(I)$ is invertible that is enatailed in the use of Ando's more restrictive definition \cite{A79} of positivity of $\Phi$; i.e., $\Phi:M_n^{++}(\C) \to M_m^{++}(\C)$, and not simply 
$\Phi:M_n^{+}(\C) \to M_m^{+}(\C)$.

\end{remark}

The following corollary of Theorem~\ref{joop} is  central to the Kubo-Ando theory of operator means \cite{KA79}.

\begin{cl}  Let  $f$ be an operator convex function on $[0,\infty)$. For all positive 
 invertible operators $X,Y$ on a Hilbert space $\cH$, let $g_f(X,Y)$ be defined by \eqref{opper}.
\begin{equation}\label{low22}
(X,Y) \mapsto g_f(X,Y)
\end{equation}
is jointly operator convex. 
\end{cl}

\begin{proof} This follows from Theorem~\ref{joop} by taking $\Phi$ to be the partial trace, exactly as in the proof of Kiefer's inequality \eqref{kief}. (Note that for functions that are homogeneous of degree one, convexity is the same as subadditivity.)
\end{proof}

 \begin{cl}\label{joopcl}  Let $f$ be an operator convex function on $[0,\infty)$. For all positive 
 invertible operators $X,Y$ on a Hilbert space $\cH$ and all positive unital maps,
 $$
\tr[ Yf(Y^{-1/2}XY^{-1/2}) ] \geq \tr[ (\Phi^\dagger(Y))f((\Phi^\dagger(Y))^{-1/2}\Phi^\dagger(X)(\Phi^\dagger(Y))^{-1/2})]\ ,
$$
 \end{cl} 
 
 \begin{proof}  This follows directly from Theorem~\ref{joop}, together with Remark~\ref{joopR}, the fact that $\Phi^\dagger$ is trace preserving, and the definition of $g_f(X,Y)$.  
 \end{proof}
 
 \begin{exam}  The function $D_{BS}(Y||X) := \tr[ Y \log( Y^{1/2}X^{-1}Y^{1/2})]$ is known as {\em Belavkin-Stasewski relative entropy} \cite{BS82}. By Corollary~\ref{joopcl}, we have the corresponding version of the Data Processing Inequality:
 $$
 D_{BS}(\Phi^\dagger(Y)||\Phi^\dagger(X)) \leq D_{BS}(Y||X)\ 
 $$
 for all unital positive maps $\Phi$; see \cite{BS82} for the completely positive case. 
 \end{exam}
 
 The coda to this paper is a timeline. We have discussed many theorems that have been proven to be equivalent; in some sense 
 these theorems are the many faces of SSA. 
The discussion has not been chronological, but rather focused on the connections between the ideas that have emerged in the 60 years  since the 
1962 paper of Wigner and Yanase \cite{WY1}. The SSA conjecture of Lanford and Robinson  \cite{LR68} was published in 1968.  The 1970 paper of Araki and 
Lieb accelerated the progress. It identified the need for a generalization of the Golden-Thompson inequality to three matrices. At least Lieb and Uhlmann 
were working  on this problem, completely independently of one another. They  both realized that proving that 
$X\mapsto \tr[\exp(H+\log X)]$ is concave on $M_n^{++}(\C)$ would lead to a proof of SSA; as explained above, this is 
equivalent through the simple Lemma~\ref{lieblem} to Lieb's Triple Matrix Theorem, which, as explained in Section 4, can be used to close the gap between the partial result of Araki and Lieb and SSA. 

The parallel but partial insights of Uhlmann are recorded in the final section of
\cite{Uh73}, in which he proved a mild generalization of Theorem~1 in  \cite{AL}, and then he conjectured \cite[Vermutung III]{Uh73}  the concavity of
$X\mapsto \tr[\exp(H+\log X)]$. Presumably, by the time he made this conjecture, Lieb had already proved it, but in any case when Lieb proved the result, he was unaware of any such conjecture by anybody else, as he has told me. Uhlmann \cite{Uh73} did not explain how the concavity of $X\mapsto \tr[\exp(H+\log X)]$ would lead to the proof of SSA.  He also noted, again without explanation,  that his Conjecture III 
(Lieb's Theorem \ref{explog}) would follow from the generalization of the Wigner-Yanase result that became Epstein's Theorem. While he did not prove his 
conjecture, he was following the trail blazed by Araki and Lieb \cite{AL} and had realized the connection with the work of  Wigner and Yanase \cite{WY1,WY2}. However, 
Lieb \cite{L73} was  the first to seal the connection between the work of Wigner and Yanase and the SSA conjecture, and to prove the missing 
Triple Matrix Theorem, Theorem~\ref{TripleMatrix}.  All of the many deep results \cite{L73} are derived by simple arguments from the Lieb Concavity Theorem, Theorem~\ref{L1}, which had settled the Wigner Yanase Dyson conjecture.  In this sense, all of the deep theorems in \cite{L73} owe their  depth to the Lieb Concavity Theorem.

The 1973 paper of Lieb \cite{L73} opened the floodgates.  The 1973 proof of SSA by Lieb and Ruskai quickly followed. By what has been explained in the last paragraph, all of the results in their paper 
ultimately depended on the Lieb Concavity Theorem. The paper \cite{L73} was written while Elliott Lieb was visiting I.H.E.S., and after 
Henri Epstein, then a Professor at I.H.E.S., learned of the results, he provided not only a proof of Theorem~\ref{Ep73},  the one  conjecture in \cite{L73}, 
but also a new proof of  the concavity of
$X\mapsto \tr[\exp(H+\log X)]$, Theorem~\ref{explog}, and his proof was independent of the Lieb Concavity Theorem.  Epstein does not discuss SSA, 
but with his proof of Theorem~\ref{explog}, one has the first route to the proof of SSA, 
building on \cite{AL}, that does not depend upon the Lieb Concavity Theorem; Theorem~\ref{L1}.  

The next important breakthrough came  in 1977  and  is due to Uhlmann \cite{Uh77}. He proved, building on work of  Pusz and Woronowicz \cite{PW75},
that the DPI was valid not only for all completely positive trace preserving  maps, but  for all  adjoints of unital Schwarz maps. He also introduced the monotonicity version of the Lieb Concavity Theorem, and even proved it for the wider class of unital Schwarz maps \cite[Proposition 17]{Uh77}. 

The next progress came in 1978 with the work of Pusz and Woronowicz \cite{PW78} who gave a new proof of the Lieb Concavity Theorem 
and the joint convexity of the relative entropy by displaying the relevant  functionals as explicit Legendre transforms.  They did not discuss SSA directly, 
but their results provided an alternate route to SSA on account of the equivalence with joint convexity of the relative entropy which was well understood by then.  The approach of Pusz and Woronowicz
was greatly simplified in 1986  by Donald \cite{Do86} who provided the first truly elementary proof of the joint convexity of the relative entropy, 
using nothing but direct computation. Again, he didnot discuss SSA, but his results provided the first completely elementary   path to 
SSA.  Finally, in \cite{CL99}, a proof of  classical SSA  that does not refer 
to conditional probabilities was provided, and it was shown, using Epstein's Theorem, that this proof extended to the quantum case. 

As noted above, since the work of Uhlmann, it has been known that  complete positivity was not required for the DPI; it is valid for all trace preserving Schwarz maps. In \cite{P03}  Petz asked an important  question: Is the DPI
valid for all positive maps trace preserving maps?  It was recently proved by M\"uller-Hermes and Reeb \cite{MHR17} that this is the case. They drew on work of Beigi \cite{B13} who proved the monotonicity under quantum operations  of certain sandwiched R\'enyi entropies that are discussed in Section 8 of this paper. His method relied on interpolation for certain  non-commutative $L^p$  spaces. Although he invoked complete positivity in his proof, M\"uller-Hermes and Reeb noted that positivity alone would have sufficed, and carried this through, and also in an infinite dimensional setting. Shortly afterwards, Jencova \cite{J18} extended the proof to a general von Neumann algebra setting.


\begin{ack}I thank Elliott Lieb for many stimulating discussions of these topics and more over many years. 
I thank Alexander M\"uller-Hermes for helpful discussion on the extension of monotonicity inequalities beyond the case of $2$-positive maps 
to the case of Schwarz maps, or even merely positive maps.  I thank Frank Hansen for helpful discussions on operator monotonicity. I am very grateful to an anonymous referee who found many 
typographical errors in the first draft, and who also made substantive suggestions for improvement that heve benefited the paper. Finally, I thank Rupert Frank for another very careful reading, and helpful suggestions after that. 
\end{ack}

\begin{funding}
This work was partially supported by  U.S.
National Science Foundation grant  DMS 2055282
\end{funding}


\end{document}